\documentclass[11pt]{article}
\usepackage{authblk}

\usepackage{fullpage}
\usepackage{layout}
\usepackage{multirow}

\usepackage{amsfonts}
\usepackage{amsmath}
\usepackage{amsthm}
\usepackage{amssymb} 

\usepackage{nicefrac}
\usepackage{mathtools}
\newcommand{\thup}{^{\text{th}}}

\usepackage{hyperref}
\usepackage{cleveref}

\usepackage{mdframed}
\usepackage{thmtools}

\definecolor{shadecolor}{gray}{0.90}
\declaretheoremstyle[
headfont=\normalfont\bfseries,
notefont=\mdseries, notebraces={(}{)},
bodyfont=\normalfont,
postheadspace=0.5em,
spaceabove=1pt,
mdframed={
  skipabove=8pt,
  skipbelow=8pt,
  hidealllines=true,
  backgroundcolor={shadecolor},
  innerleftmargin=4pt,
  innerrightmargin=4pt}
]{shaded}

\declaretheorem[style=shaded,within=section]{definition}
\declaretheorem[style=shaded,sibling=definition]{theorem}
\declaretheorem[style=shaded,sibling=definition]{proposition}

\declaretheorem[style=shaded,sibling=definition]{corollary}

\declaretheorem[style=shaded,sibling=definition]{lemma}

\usepackage{algorithm}
\usepackage[noend]{algpseudocode}

\algrenewcommand\algorithmicwhile{\textbf{While}}

\usepackage{caption}
\usepackage{subcaption}

\usepackage{xcolor}
\usepackage{color}
\usepackage{graphicx}
\graphicspath{{./figures/}}  

\usepackage[normalem]{ulem}

\newcommand{\R}{\mathbb{R}} 
\newcommand{\N}{\mathbb{N}} 

\newcommand{\cD}{{\cal D}}

\newcommand{\cO}{{\cal O}}

\newcommand{\mA}{{\bf A}}
\newcommand{\mB}{{\bf B}}

\newcommand{\mD}{{\bf D}}

\newcommand{\mF}{{\bf F}}

\newcommand{\mH}{{\bf H}}
\newcommand{\mI}{{\bf I}}

\newcommand{\mK}{{\bf K}}

\newcommand{\mM}{{\bf M}}

\newcommand{\mS}{{\bf S}}

\newcommand{\mW}{{\bf W}}
\newcommand{\mX}{{\bf X}}

\newcommand{\mZ}{{\bf Z}}

\usepackage[colorinlistoftodos,bordercolor=orange,backgroundcolor=orange!20,linecolor=orange,textsize=scriptsize]{todonotes}

\newcommand{\eqdef}{\coloneqq}

\newcommand{\ve}[2]{\langle #1 ,  #2 \rangle} 
\newcommand{\dotprod}[1]{\left< #1\right>} 
\newcommand{\norm}[1]{ \| #1 \|}      

\newcommand{\Prob}[1]{\mathbb{P}[#1]}
\DeclareMathOperator{\Range}{Range}     

\DeclareMathOperator{\argmin}{argmin}        

\DeclarePairedDelimiter\floor{\lfloor}{\rfloor}
\DeclarePairedDelimiter\ceil{\lceil}{\rceil}

\providecommand{\range}[1]{{\rm Range}\left( #1\right)}

\providecommand{\trace}[1]{{\rm Trace}\left( #1\right)}

\newcommand{\E}[1]{\mathbb{E}\left[#1\right] }

\newcommand{\eg}{{\em e.g.,~}}

\usepackage[newfloat,frozencache=true,cachedir=minted-cache]{minted}

\SetupFloatingEnvironment{listing}{name=Code}
\usepackage[section]{placeins}
\usepackage{listings}

\usepackage{tcolorbox}
\BeforeBeginEnvironment{minted}{\begin{tcolorbox}}%
\AfterEndEnvironment{minted}{\end{tcolorbox}}%

\usepackage[giveninits=true,backend=bibtex,style=alphabetic,citestyle=alphabetic, url =false, arxiv =false, isbn = false, doi = false]{biblatex}
\usepackage{wrapfig}
\bibliography{biblio}

\begin{document}

\title{\texttt{RidgeSketch}: A Fast sketching based solver for large scale ridge regression}

\author[1]{Nidham Gazagnadou}
\author[2]{Mark Ibrahim}
\author[1]{Robert M. Gower}
\affil[1]{LTCI, T\'el\'ecom Paris, Institut Polytechnique de Paris \protect\\ email: \textit{\{nidham.gazagnadou,robert.gower\}@telecom-paris.fr}}
\affil[2]{Facebook AI Research \protect\\ \text{email: \textit{marksibrahim@fb.com}}}

\maketitle

\begin{abstract}
	We propose new variants of the sketch-and-project method for solving large scale ridge regression problems.
	Firstly, we propose a new momentum alternative and provide a theorem showing it can speed up the convergence of sketch-and-project, through a fast \emph{sublinear} convergence rate.
	We carefully delimit under what settings this new sublinear rate is faster than the previously known linear rate of convergence of sketch-and-project without momentum.
	Secondly, we consider combining the sketch-and-project method with new modern sketching methods such as the count sketch, subcount sketch (a new method we propose), and subsampled Hadamard transforms.
	We show experimentally that when combined with the sketch-and-project method, the (sub)count sketch is very effective  on sparse data and the standard subsample sketch is effective on dense data.
	Indeed, we show that these sketching methods, combined with our new momentum scheme, result in methods that are competitive even when compared to the Conjugate Gradient method on real large scale data.
	On the contrary, we show the subsampled Hadamard transform does not perform well in this setting, despite the use of fast Hadamard transforms, and nor do recently proposed acceleration schemes work well in practice.
	To support all of our experimental findings, and invite the community to validate and extend our results, with this paper we are also releasing an open source software package: \texttt{RidgeSketch}.
	We designed this object-oriented package in Python for testing sketch-and-project methods and benchmarking ridge regression solvers.
	\texttt{RidgeSketch} is highly modular, and new sketching methods can easily be added as subclasses. We provide code snippets of our package in the appendix.
\end{abstract}


\section{Introduction}

Consider the regression problem given by
\begin{equation}
	\label{eq:ridge}
	\min_{w \in \R^d} \frac{1}{2}\norm{\mX w - y}_2^2 + \frac{\lambda}{2}\norm{w}_2^2 \enspace,
\end{equation}
where $\mX \in \R^{n \times d}$  is the feature matrix, $y \in \R^n$ the targets and $\lambda >0$ the regularization parameter.

The need to solve linear regression such as~\eqref{eq:ridge} occurs throughout scientific computing, where it is known as ridge regression in the statistics community~\cite{sgv-rrladv-98,vovk2013kernel}, data assimilation in weather forecasting~\cite{Krishnamurtiweather} and linear least squares with Tikhonov regularization among numerical analysts~\cite{GolubTikhonov}.

Here we focus on applications where the number of rows and columns of $\mX \in \R^{n \times d}$ are both large scale.
With datasets now being collected automatically and electronically, there has been a surge of new stochastic incremental methods that can gracefully scale with the dimensions of the data in~\eqref{eq:ridge}. Yet, it is still unclear if the new stochastic methods are capable of outperforming the classic Conjugate Gradient (CG) methods~\cite{Hestenes1952}, even though the CG method was developed in the 1950's. We address this question, by designing and implementing new stochastic methods based on randomized sketching~\cite{Pilanci2015,Pilanci2014} together with the iterative projection methods known as the sketch-and-project methods~\cite{Gower2015}.
Our main objective can be stated simply as

\begin{quote}
	{\bf Main objective:}  Develop variants of the sketch-and-project method that are competitive in practice, and provide new theoretical support for these variants.
\end{quote}

We do this by exploring the use of new sketching transforms such as  \texttt{Count sketch}~\cite{CountSketch2002} and proposing a new momentum variant of the sketch-and-project method.
To demonstrate the benefits of our new momentum and sketching variant, we show how our resulting method is competitive with the Conjugate Gradients (CG) methods, and supersedes previous momentum based variants~\cite{Loizou2020mom}.
We will also give some several negative results, such as showing how, in our setting, recently developed accelerated variants~\cite{TuVWGJR17,acellqN2018} are not viable in practice, and nor are the fast Johnson-Lindenstrauss (JL) transforms~\cite{Ailon:2009}.

In the following~\Cref{sec:background_contrib} we outline some of the background and state our main contributions.
In~\Cref{sec:sketch_and_project}, we describe the sketch-and-project method for solving ridge regression and the different sketching methods we explored in~\Cref{sec:sketch_matrices}.
Then, we give a convergence theory for our method without and with momentum respectively in~\Cref{sec:conv_theory,sec:mom}, that we specialize in~\Cref{sec:specialconv} for single column sketches.
To illustrate this theoretical results, we provide large scale numerical experiments highlighting our theoretical findings in~\Cref{sec:mom_exp} and showing that our momentum sketch-and-project method competes against CG and direct solvers.
Finally, in~\Cref{sec:accel} we present the accelerated version of our method and show its inefficiency in practice since it requires additional overhead costs and spectral information of $\mA$ to compute acceleration parameters.

\section{Background and contributions}
\label{sec:background_contrib}

\paragraph{Iterative sketching in linear systems.}
When the dimensions $n$ and $d$ are large, direct methods for solving~\eqref{eq:ridge} can be infeasible, and iterative methods are favored.
In particular, the Krylov methods including the CG algorithm~\cite{Hestenes1952} are the industrial standard so long as one can afford full matrix vector products and the system matrix fits in memory.
On the other hand, if a single matrix vector product is considerably expensive, or the data matrix is too large to fit in memory, then iterative methods exploiting only few rows or columns of $\mA$ are effective.
This includes, for example, randomized Kaczmarz methods~\cite{Kaczmarz1937,Strohmer2009,necoara2019faster}, its greedy or deterministic variants~\cite{petra2016single,du2019new,de2017sampling,haddock2021greed,bai2018greedy}, and Coordinate Descent (CD) method~\cite{Leventhal2010,Ma2015,wright2015coordinate}.
In~\cite{Gower2015}, the authors united Kaczmarz, CD and host of other randomized  iterative methods under the sketch-and-project framework.\\[0.2cm]

\noindent{\emph{Contributions}.} We revisit the sketch-and-project method and examine how to make specialized variants for solving ridge regression that are competitive as compared to conjugate gradients.
To do so, we are also releasing a high-quality and modular Python package called \texttt{RidgeSketch}\footnote{Our fully documented and tested package is available at \url{https://github.com/facebookresearch/RidgeSketch}.} for efficiently implementing and testing sketch-and-project methods.
More information about the code and how to contribute by adding sketches or datasets are detailed in~\Cref{sec:code}.

\paragraph{Momentum.}
Recently, a variant of sketch-and-project with momentum was proposed in~\cite{Loizou2020mom}.
The authors of~\cite{Loizou2020mom} show a theoretical advantage in terms of convergence in expectation (of the first moment only),
but do not present any advantage in terms of convergence in L2 or high probability. There has even more recent work~\cite{Morshed:2020-SSD} analysing momentum together with sketch-and-project, and extensions to solving the linear feasibility problem~\cite{MorshedIN20,Morshed-2020-feasibility}. Yet none of these works show any theoretical benefit of using momentum over using no momentum at all.
This lack of a theoretical benefit of using momentum is echoed throughout stochastic optimization: there are no almost no theoretical benefits in using momentum outside of the strongly convex full batch setting~\cite{Polyak64}. \\[0.2cm]

\noindent{\emph{Contributions}.}  Here we have establish the first theoretical advantage of using momentum together with sketch-and-project method. We show that when using sketch-and-project with momentum, the last iterate enjoys a fast sublinear convergence. Without momentum, this result does not hold. Instead, without momentum, it has only been shown to hold for the average of the iterates. As depicted in~\Cref{fig:superiority_sublinear_over_linear},
we show that the fast sublinear rate of our new momentum variant gives a tighter complexity bound than the previously best known linear rate of convergence~\cite{SDA} when the desired precision is moderate and the condition number of the underlining problem is moderate to large.
Our new convergence theory for momentum also suggests completely iteration dependent  schedules for setting the momentum parameter. We perform extensive numerical tests showing the superiority of this new scheduling as compared to using a constant momentum.

\begin{figure}
	\centering
	\includegraphics[width=.5\textwidth]{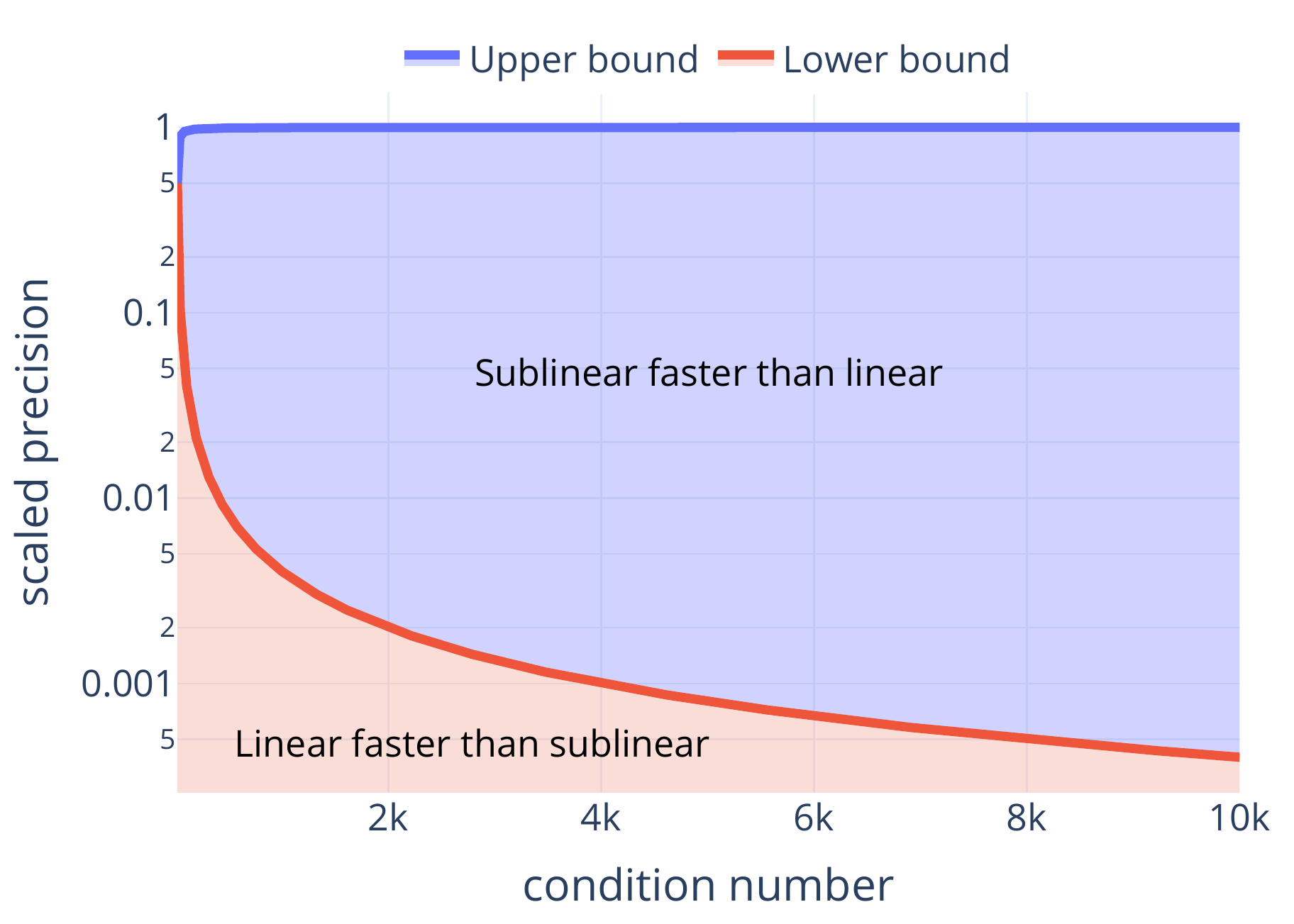}
	\caption{Areas of superiority of sublinear convergence of CD with momentum (in blue) and linear convergence of CD without momentum (in red), where the upper and lower bounds are respectively the left and right-hand side of~\eqref{eq:bounds_scaled_eps}.}
	\label{fig:superiority_sublinear_over_linear}
\end{figure}

\paragraph{Acceleration.}
Lee and Sidford~\cite{Lee2013} show how the Kaczmarz method could be accelerated through its connection to CD, and also compare the resulting accelerated convergence rate to the rate of the Conjugate Gradient algorithm.
Though, they provide no numerical experiments leaving it unclear if this form of acceleration can afford any practical advantage.
Later, Liu and Wright~\cite{LiuWright-AccKacz-2016} developed an accelerated Kaczmarz method and show that it can be faster than CG on densely generated artificial data.
But they do not provide examples of this on real data or an affordable rule for setting the acceleration parameters.
More recently, it was shown that the entire family of the sketch-and-project methods could be accelerated~\cite{acellqN2018}.
Yet, experiments of the authors rely on a grid search that would defeat any gains in using acceleration.  \\[0.2cm]

\noindent{\emph{Contributions}.}  We investigate the possibility of developing a practical setting for the two acceleration parameters proposed in~\cite{acellqN2018} that would result in a robust performance gain over standard sketch-and-project.
Unlike Nesterov's acceleration for gradient descent in the convex setting, there are no default parameter settings that work consistently across a significant class of problems.
We show through a careful grid search that finding good parameters is like ``looking for a needle in the haystack'' and that identifying any practical settings for these parameters is virtually impossible.
We thus recommend that, to advance the use of acceleration for linear systems, one would need to study a smaller class of problems and certain spectral bounds to derive a working rule for setting the parameters.

\paragraph{Johnson-Lindenstrauss Sketches}
In~\cite{JL}, the authors show how high dimension data can be projected onto a low dimension subspace using a Gaussian matrix in such a way that it approximately preserves the pairwise distance between points.
This result is now known as the celebrated Johnson-Lindenstrauss (JL) Lemma. Since then, many more random transforms have been shown to satisfy this property, such as the Count sketch~\cite{Cormode2003}, subsampled Fourier~\cite{Ailon:2009} and Hadamard transforms~\cite{boutsidis2013improved,tropp2011improved,JayramW13}.
These JL \emph{transforms} have been used to speed up LAPACK solvers~\cite{Maymounkov2013}, solving linear regression~\cite{WangGM17} and Newton based methods~\cite{Pilanci2014}.\\[0.2cm]

\noindent{\emph{Contributions}.} We propose new combinations of the sketch-and-project method with JL sketches such as the Subsampled Randomized Hadamard Transform (SRHT) and a new subsampled count (SubCount) sketch. Our new package \texttt{RidgeSketch} is also setup in a way that is easily extensible, where new sketching methods can easily be added as a new instance of a \emph{sketching class} (see~\Cref{sec:code}).
Our results for the subsampled Hadamard sketch are negative. We show, despite the favourable theoretical complexity of using Hadamard sketches,  that the overhead costs make them a completely impractical choice for sketch-and-project methods. The SubCount sketch, on the other hand, when combined with sketch-and-project, results in an efficient method.

\paragraph{Alternative Iterative Sketching based methods.}
A closely related method to the sketch-and-project method is the Iterative Hessian Sketch~\cite{Pilanci2014}, which makes use of iterative sketching to solve constrained quadratics. In the unconstrained setting such as~\eqref{eq:ridge},   Iterative Hessian Sketch is efficient when the  dimension $d$ to be significantly smaller than number of data points $n$.
This  rules out one of our main applications: kernel ridge regression where $n=d$.\\[0.2cm]

\subsection{Linear system formulation}
\label{sec:ridgelinear}

Since the optimization problem~\eqref{eq:ridge} is differentiable and without constraints, its solution satisfies the stationarity conditions given by
\begin{equation}
	\label{eq:linear}
	\left(\mX^\top \mX +\lambda \mI \right) w = \mX^\top y \enspace.
\end{equation}
We can also rewrite the above linear system in its \emph{dual form} given by
\begin{equation}
	\label{eq:duallinear}
	w = \mX^\top \alpha \enspace, \quad \mbox{where} \quad \left(\mX \mX^\top +\lambda \mI \right) \alpha = y \enspace.
\end{equation}
The equivalence between solving~\eqref{eq:linear} and~\eqref{eq:duallinear} is well known and proven in the appendix in Lemma~\ref{lem:dual_equiv} for completion.
The choice of solving~\eqref{eq:linear} or~\eqref{eq:duallinear} will depend on the dimensions of the data.
On the one hand, the system in~\eqref{eq:linear} involves a $d \times d$ matrix, and thus the \emph{primal form}~\eqref{eq:linear} is preferred when $d \leq n$.
On the other hand, the system in~\eqref{eq:duallinear} involves a $n \times n$ matrix, and thus solving the \emph{dual form}~\eqref{eq:duallinear} is preferred when $n < d$.

In either case, the bottleneck cost is the solution of a linear system where the system matrix is symmetric and positive definite. To simplify notation we introduce $\mA \eqdef \mB + \lambda \mI$, and let
\begin{equation}
	\label{eq:Axb}
	\mA w =b \enspace,
\end{equation}
where $\mB = \mX^\top \mX$ and $b = \mX^\top y$, for the primal form~\eqref{eq:linear}, or $\mB = \mX \mX^\top$ and $b=y$ for the dual one~\eqref{eq:duallinear}. Note that in the later case, a multiplication by $\mX^\top$ is required to recover the weights vector.
Let $m$ be the dimensions of $\mA \in \R^{m\times m}$, thus $m$ equals the smallest dimension of the design matrix $\mX \in \R^{n\times d}$. Indeed, $m = d$ if we choose to solve the primal version~\eqref{eq:linear} and $m=n$ if we choose to solve the dual one~\eqref{eq:duallinear}.

Thus solving the ridge regression problem boils down to finding the solution of the linear system~\eqref{eq:Axb}. If the dimension of the squared matrix $\mA$, denoted $m$, is not too large, one can solve this problem using a direct solver (for instance through a SVD or a Cholesky decomposition).
But when $m$ is large, direct methods become intractable as their computational cost grows with $\cO(m^3)$.

\subsection{Using a kernel}

We also consider kernel ridge regression, where the feature matrix is the result of applying a feature map. This leads to particular considerations since the resulting feature matrix may have an infinite number of columns.

The idea behind kernel ridge regression is that, instead of learning using the original input (or feature) vectors $\mX \eqdef [x_1,\ldots, x_n]$, we can learn using a high dimensional feature map of the inputs
$\phi : \R^d \rightarrow \R^r$
where $r >d$ or even an infinite dimensional space. For instance, $\phi(x)$ could encode a high dimensional polynomial.  By replacing each $x_i$ with $\phi(x_i)$ in~\eqref{eq:ridge} we arrive at
\begin{equation}
	\label{eq:ridge_feature}
	\min_{w \in \R^r} \frac{1}{2} \sum_{i=1}^n(\phi(x_i) w - y_i)^2 + \frac{\lambda}{2}\norm{w}_2^2 \enspace.
\end{equation}
When $r$ is large, or even infinite, solving~\eqref{eq:ridge_feature} directly can be difficult or intractable.
Fortunately, the \emph{dual formulation} of~\eqref{eq:ridge_feature} is always an $n$--dimensional problem independently of the dimension $r$~\footnote{This is commonly known as the  \emph{kernel trick}, see Chapter 16 in~\cite{Shalev-Shwartz:2014book}.}. The \emph{dual formulation} of~\eqref{eq:ridge_feature} is given by
\begin{equation}
	\label{eq:dual_ridge_feature}
	\alpha^* =\arg\min_{\alpha \in \R^n} \frac{1}{2} \norm{\mK \alpha - y}_2^2 + \frac{\lambda}{2}\alpha^\top \mK \alpha \enspace,
\end{equation}
where $\mK = \left(\mK(x_i, x_j)\right)_{ij} \eqdef \dotprod{\phi(x_i), \phi(x_j)}_{ij}$ is the kernel matrix. This is equivalent to solving in $\alpha$ the linear system
\begin{equation}\label{eq:kernelridge}
	\mK \left(\mK + \lambda \mI \right)\alpha = \mK y \enspace.
\end{equation}
With $\alpha^*$, the solution to the above, we can then predict the output of a new input vector $x$ using
\begin{equation*}
	\mbox{predict}(x) \; = \;  \sum_{i=1}^n \alpha_i^*\dotprod{\phi(x_i),\phi(x) } \; = \;  \sum_{i=1}^n \alpha_i^*\mK(x_i,x) \; = \; \mK \alpha^* \enspace.
\end{equation*}
Consequently to solve~\eqref{eq:dual_ridge_feature} and make predictions, we only need access to the kernel matrix. Fortunately, there are several feature maps for which the kernel matrix is easily computable including the one we use in our experiments which is the Gaussian Kernel, otherwise known as the Radial Basis Function
\begin{equation}
	\label{eq:Gausskernel}
	K(x, x') = \exp\left(\frac{-||x-x'||^2}{2\sigma^2}\right) \enspace.
\end{equation}
where $\sigma>0$ is the \emph{kernel} parameter.

Ultimately, despite the addition of a kernel, the resulting problem~\eqref{eq:kernelridge} is still a linear system  of the form $\mA w = b$ where
\begin{equation*}
	\mA =  \mK^\top(\mK + \lambda \mI) \quad \mbox{and}\quad b = \mK^\top y \enspace.
\end{equation*}
The only marked difference now is that $\mA$ tends to be dense, and because of this, matrix-vector products are particularly expensive.

\section{The Sketch-and-Project method}
\label{sec:sketch_and_project}

Sketch-and-project is an archetypal algorithm that unifies a variety of randomized iterative methods including both randomized Kaczmarz and CD~\cite{Gower2015}, and all their block and importance sampling variants.

At each iteration, the sketch-and-project methods randomly compressed the linear system using what is known as a \emph{sketching matrix}.
\begin{definition}
	Let $\tau \in \N$ and let $\cD$ be a distribution over matrices in $\R^{m \times \tau}$. We refer to $\tau$ as the sketch size and to $\mS \in \R^{m \times \tau}$ drawn from the distribution $\cD$ as a sketching matrix.
\end{definition}

We can use a sketching matrix $\mS$ to reduce the number of rows of the linear system~\eqref{eq:Axb} to $\tau$ rows as follows
\begin{equation}
	\label{eq:sketchedAxb}
 	\mS^\top \mA w = \mS^\top b \enspace.
\end{equation}
If the sketch size $\tau$ is sufficiently large and the sketching matrix is appropriately chosen, then we can guarantee with high probability that the solution to the \emph{sketched} linear system~\eqref{eq:sketchedAxb}
is close to the solution $w^*$ of the original system~\eqref{eq:Axb} (see~\cite{Mahoney2011}). But this \emph{one-shot} sketching approach poses several challenges 1) it may be hard to determine how large $\tau$ should be, 2) the sketched linear system now has multiple solutions and 3) with some low probability the solution to~\eqref{eq:sketchedAxb} could be far from $w^*.$ To address these issues, we use an iterative projection scheme.

\begin{algorithm}
	\begin{algorithmic}[1]
		\State \textbf{Parameters:} distribution over random $m \times \tau$ matrices in $\cD$, tolerance $\epsilon > 0$
		\State Set $w^0 = 0 \in \R^m$ \Comment weights initialization
		\State Set $r^0 = \mA w^0 - b = - b \in \R^m$ \Comment residual initialization
		\State $k = 0$
		\While {$\norm{r^k}_2 / \norm{r^0}_2 \leq \epsilon$}
			\State Sample an independent copy $\mS_k \sim \cD$
			\State $r_{\mS}^k = \mS_k^\top r^k$ \label{ln:sketchrhs} \Comment compute sketched residual
			\State $\delta_k =$ \texttt{least\_norm\_solution} $\left( \mS_k^\top \mA \mS_k, r_{\mS}^k \right)$  \label{ln:algo_sketched_system} \Comment solve sketched system
		  	\State $w^{k+1} =  w^k - \mS_k \delta_k$ \Comment update the iterates
			\State $r^{k+1} = r^k - \mA \mS_k \delta_k$ \Comment update the residual
			\State $k = k + 1$
		\EndWhile
		\State \textbf{Output:} $w^{t}$ \Comment return weights vector
	\end{algorithmic}
	\caption{The Sketch-and-Project method}
	\label{alg:SP}
\end{algorithm}

Let $\mW$ be a symmetric positive definite matrix of order $n$ (which will typically be chosen as $\mA$), here we project with respect to the $\mW$--norm\footnote{Using this norm for symmetric positive definite matrices has shown to result in algorithms with a fast convergence rate~\cite{Gower2015}.} given by $\norm{\cdot}_{\mW} = \sqrt{\ve{\cdot}{\mW \cdot}}$.

At the $k\thup$ iteration of the sketch-and-project algorithm,  a sketching matrix $\mS_{k}$ is drawn from $\cD$ and the current iterate $w^{k}$ is projected onto the solution space of the sketched system $\mS_{k}^\top \mA x = \mS_{k}^\top b$ with respect to the $\mW$--norm, that is
\begin{equation}
	\label{eq:NF}
	w^{k+1} = \argmin_{w\in \R^m} \norm{w- w^{k}}_\mW^2 \quad \mbox{subject to} \quad \mS_{k}^\top \mA w = \mS_{k}^\top b \enspace.
\end{equation}
The closed form solution to~\eqref{eq:NF} is given by
\begin{equation}
	\label{eq:xupdateW}
	w^{k+1} = w^k - \mW^{-1} \mA \mS_k \left(\mS_k^\top \mA \mW^{-1} \mA \mS_k \right)^{\dagger} \mS_k^\top(\mA w^k-b) \enspace,
\end{equation}
where ${}^\dagger$ denotes the pseudoinverse. When $\mA$ is known to be positive definite, as in our case, using $\mW = \mA$ often results in an overall faster convergence of~\eqref{eq:xupdateW} as shown in~\cite{Gower2015}. Using  $\mW = \mA$  in~\eqref{eq:xupdateW}  gives the updates
\begin{equation}
	\label{eq:xupdate}
	w^{k+1} = w^k - \mS_k \left( \mS_k^\top \mA \mS_k \right)^{\dagger} \mS_k^\top(\mA w^k-b) \enspace.
\end{equation}
We refer to~\eqref{eq:xupdate} as the \texttt{RidgeSketch} method since it is specialized for solving ridge regression.
Here we give the details on how to efficiently implement the \texttt{RidgeSketch} update~\eqref{eq:xupdate}, see~\Cref{alg:SP}.

One practical detail we have added to the pseudocode in~\Cref{alg:SP} is a stopping criteria.
For any iterative algorithm, it is important to know when to stop.
We can do this by monitoring the \emph{residual} $r^k \eqdef \mA w^k - b$.
From an initial residual $r^0 \in \R^m$, when the relative residual $\norm{r^k} / \norm{r^0}$ is below a given tolerance, we stop. We also need this residual for computing the update~\eqref{eq:xupdate}. We can efficiently update the residual from one iteration to the next since
\begin{equation}
	r^{k+1} = \mA w^{k+1} - b
	= \mA (w^k - \mS_k  \delta_k) - b
	= r^k - \mA \mS_k \delta_k \enspace,
\end{equation}
where $\delta_k \eqdef \left( \mS_k^\top \mA \mS_k \right)^{\dagger} \mS_k^\top(\mA w^k-b) = \left( \mS_k^\top \mA \mS_k \right)^{\dagger} \mS_k^\top r^k$.
Thus we can update the residual at the cost of $O(m \tau)$, that is, multiplying a $m \times \tau$ matrix  $\mA \mS_k$ with the $\tau$ dimensional vector $\delta_k$. Note that $\delta_k$ is can be efficiently computed as the least-norm solution of the following linear system in $x$
\begin{equation}
	\label{eq:sketched_system}
	\mS_k^\top \mA  \mS_k x = \mS_k^\top r^k \enspace.
\end{equation}

\section{Sketching methods and matrices}
\label{sec:sketch_matrices}

Here we introduce several sketching matrices that can be used in Algorithm~\ref{alg:SP}. The ideal sketch is one that reduces the dimension of the linear system~\eqref{eq:Axb} as much as possible, while preserving as much information as possible and that can be efficiently implemented.  As we discuss throughout this section, there is no sketch that has all three of these qualities, and ultimately, one must make a trade-off between them.

\subsection{Classical sketches}

One of the most classical and simple sketching method is the \emph{Gaussian sketch}.
\begin{definition}
	A {\bf Gaussian sketch} is a random matrix $\mS \in \R^{m \times \tau}$ where is each element is sampled i.i.d for the standard Gaussian distribution.
\end{definition}

As argued in~\cite{Pilanci2015}, the resulting sketched matrix $\mS^\top \mA$ can be a good approximation to the full matrix and is easy to control with probability bounds. Though simple to implement, the cost of forming $\mS^\top \mA$ is $\cO(\tau m^2)$, which is expensive.

A much cheaper option is to use a \emph{Subsampling sketch}.
\begin{definition}
	A {\bf Subsampling sketch} is based on a randomly sampled subset $C \subset \{1, \ldots, m\}$  with $|C| = \tau$ elements drawn uniformly on average from all such subsets. Let $ \mI_ {C} \in \R^{m \times \tau}$ denote the concatenation of the columns of the identity matrix $\mI \in \R^{m \times m}$ whose columns are indexed in $C$. We define the subsampling sketch distribution as
	\begin{equation}
		\label{eq:subsample}
		\Prob{\mS = \mI_{C}} \; = \; \frac{1}{\binom{m}{\tau}} \enspace, \quad \mbox{for all}\quad C \subset \{1, \ldots, m\}, \; |C| = \tau \enspace.
	\end{equation}
\end{definition}

Subsampling sketches are very cheap to compute, indeed, we need not even compute $\mS^\top \mA$ since is simply equivalent to fetching the rows of $\mA$ indexed by a random subset. This can be done in Python without generating any copies of the data by slicing the selected rows.
Slicing is very well optimized operation in NumPy~\cite{van2011numpy} and SciPy~\cite{virtanen2020scipy} Compressed Sparse Row (CSR) sparse arrays, which makes it one the fastest sketching method.

Though subsampling sketches are cheap and fast, the sketched matrix $\mS^\top \mA$ can be a poor approximation of $\mA$, since it is always possible that some vital part of $\mA$ is ``left out'' in the rows that were not sampled. Still, the subsampling sketch will prove to work well within the iterative sketch-and-project scheme.

Next we consider a sketch that makes use of subsampling, addition and subtraction of rows.

\subsection{Count and SubCount sketch}

In order to avoid losing too much information by just subsampling rows, one can also  sum and subtract groups of rows.
This is the idea behind \texttt{Count sketch} which stems from the streaming data literature~\cite{CountSketch2002,CountMinSketch2005} and got popularized as a matrix sketching tool by~\cite{clarkson2017low}. Count sketch selects rows of $\mA$, flips their sign with probability $1/2$ and add it to a random row, sampled uniformly, of the output matrix $\mS^\top \mA$.

To decrease the overall cost of Count Sketch, we also combined it with a subsampling step. We call the resulting method the \texttt{SubCount sketch}.
SubCount sketch has two parameters, the subsampling size $s \in \{1,\ldots, m\} $ and the sum size $k \in \{1,\ldots, s\}$  that must be such that it divides $s$.
The SubCount sketch can be broken down into three steps: subsampling $s$ random rows of the input, then randomly flipping their sign and finally summing $k$ contiguous  rows together. This can also be illustrated in terms of matrix multiplications as follows.

\begin{figure}
    \centering
    \begin{subfigure}{0.45\textwidth}
		\centering
		\includegraphics[width=.5\textwidth]{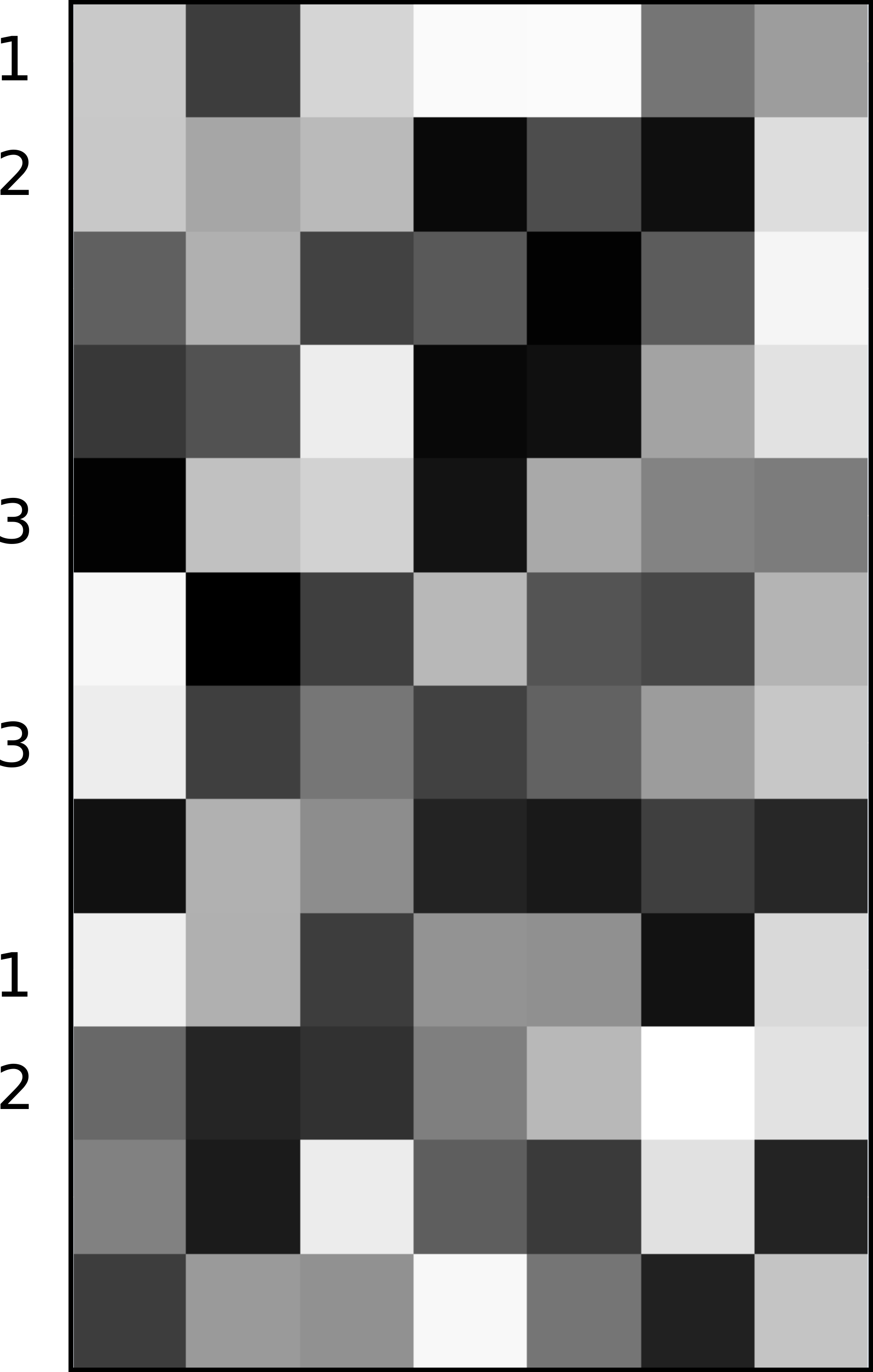}
		\caption{Selection of rows with subsampling size $s=6$.}
    \end{subfigure}
	\hspace{1cm}
    \begin{subfigure}{0.45\textwidth}
		\centering
		\includegraphics[width=.5\textwidth]{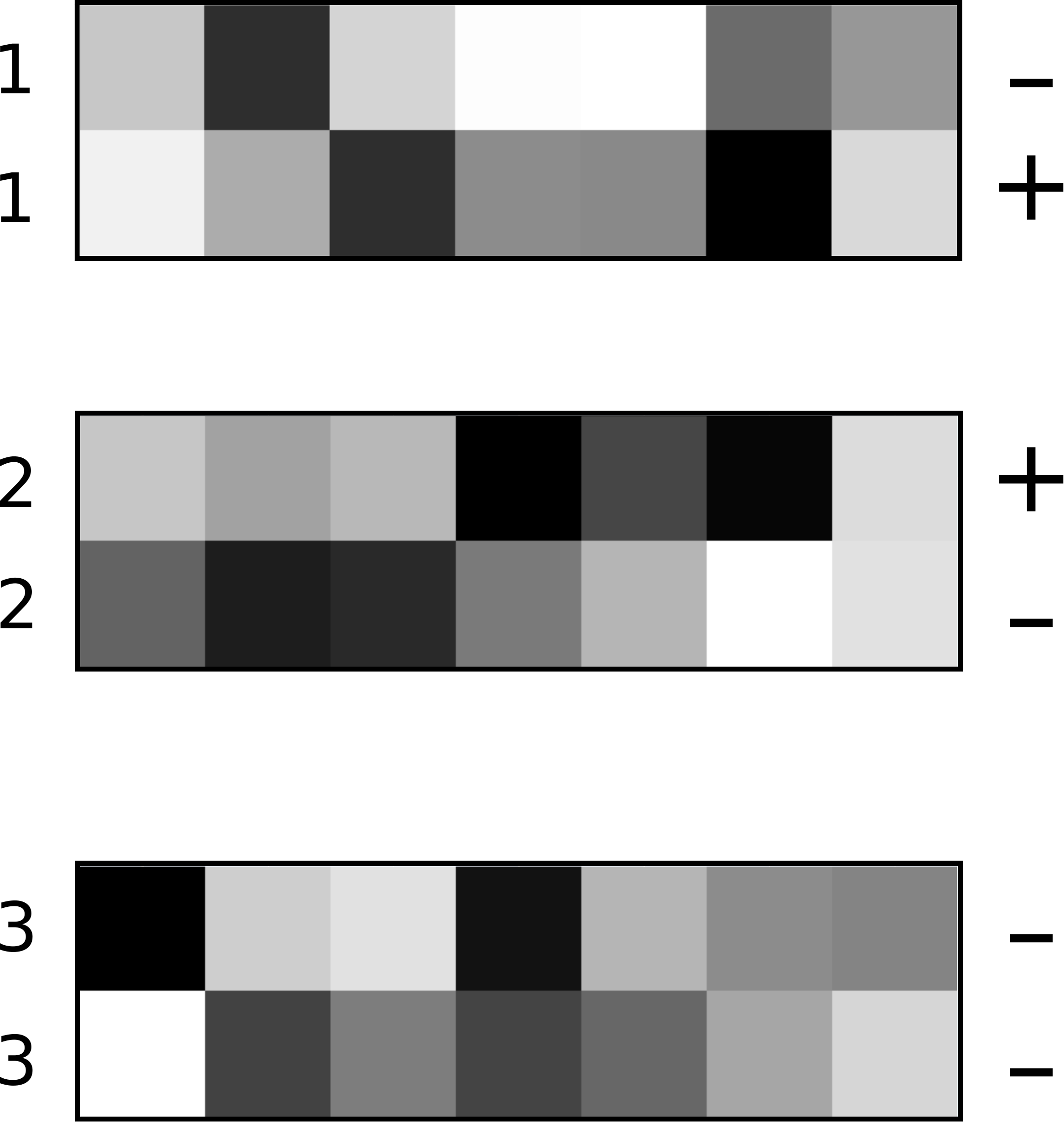}
        \caption{Random sign flip of selected rows that will be summed together with sum size $k=2$.}
    \end{subfigure}
    \caption{SubCount sketch example.}
    \label{fig:subcount}
\end{figure}

\begin{definition}
	A {\bf SubCount sketch} is a matrix $\mS \in \R^{m \times \tau}$ such that $\mS^\top = \mathbf{\Sigma} \mD \mI_C$, where
	\begin{itemize}
		\item $\mI_C \in \R^{s \times m} $ is a subsampling matrix based on a set $C \subset \{1,\ldots, m\}$ chosen uniformly at random from all sets with $s$ elements
		\item $\mD \in \R^{s\times s}$ is a diagonal matrix with elements sampled uniformly from $\{-1, 1\}$
		\item $\mathbf{\Sigma}  \in \R^{\frac{s}{k} \times s}$ is a \emph{sum} matrix, that sums every $k$ contiguous rows together, that is
		\begin{equation}\label{eq:Sigmacount}
			\mathbf{\Sigma} = \begin{bmatrix}
				\underbrace{1 \cdots 1}_{k} & 0 \cdots 0& \cdots & 0 \cdots 0 \\
				0 \cdots 0 & \underbrace{1 \cdots 1}_{k} & \cdots & 0 \cdots 0 \\
				0\cdots 0 & 0  \cdots 0& \cdots & \underbrace{1 \cdots 1}_{k}
			\end{bmatrix}
		\end{equation}
	\end{itemize}
\end{definition}

See Figure~\ref{fig:subcount} for a depiction of the first two steps of SubCount sketch. The resulting sketch size is given by $\tau = s/k \in \N^*$.
In our \texttt{RidgeSketch} package, since the sketch size $\tau$ is a parameter selected by the user, we  adjusted $s$ and $k$ such that $\tau = s/k$ as follows: If $10 \, \tau \leq m$, we set arbitrarily $k$ to $10$, else, we fix $k = \lfloor m / \tau \rfloor$ and then compute $s = k \tau$.
When when we have no subsampling, that is $s =m$, then we simply refer to the SubCount Sketch and Count Sketch\footnote{The standard definition of Count Sketch also shuffles the columns of the summing matrix~\eqref{eq:Sigmacount}. We did not use this shuffling since we found that it had little to no effect on performance in our setting.}.
The advantage of Count sketch is that it enables the computation of $\mS^\top \mA$ with a cost of $\cO (nnz(\mA))$, where $nnz(\mA)$ is the number of non-zero values of $\mA$. Thus, this sketching method is fast for large sparse matrices, especially for the CSR (column sparse rows) format for which accessing rows is efficient. Moreover, because it linearly combines groups of rows of the input matrix, it avoids the pitfall of leaving out meaningful rows of $\mA$, unlike subsampling. This will later be confirmed numerically in~\Cref{sec:mom_exp}.

\subsection{Subsampled Randomized Hadamard Transform}
\label{sec:srht}

Here we consider the \emph{Subsampled Randomized Hadamard Transform} (SRHT)~\cite{woolfe2008fast,tropp2011improved}, which we refer to as \texttt{Hadamard sketch} for short.
\begin{definition}
	\label{def:srht}
	A {\bf Hadamard sketch or SRHT} is a matrix $\mS \in \R^{m \times \tau}$ such that
	\begin{equation}
		\mS^\top = \frac{1}{\sqrt{\tau m}} \mI_C \mH_m \mD \enspace,
	\end{equation}
	where we assume there exists $q \in \N^*$ such that $m=2^q$ with, and
	\begin{itemize}
		\item $\mD \in \R^{m \times m}$ is a diagonal matrix with elements sampled uniformly from $\{-1, 1\}$
		\item $\mH_m \in \R^{m \times m}$ is the Hadamard matrix of order $m$ defined recursively through
		\begin{equation}
			\label{eq:hadamard_matrix}
			\mH_1 =
			\begin{bmatrix}
				1
			\end{bmatrix}, \quad
			\mH_2 =
			\begin{bmatrix}
				1 & 1 \\
				1 & -1
			\end{bmatrix},
			\quad \ldots, \quad
			\mH_{2^q} =
			\begin{bmatrix}
				\mH_{2^{q-1}} & \mH_{2^{q-1}} \\
				\mH_{2^{q-1}} & -\mH_{2^{q-1}}
			\end{bmatrix} \in \R^{m \times m}
		\end{equation}
		\item $\mI_C \in \R^{\tau \times m} $ is a subsampling matrix based on a set $C \subset \{1,\ldots, m\}$ chosen uniformly at random from all sets with $\tau$ elements
	\end{itemize}
\end{definition}

The Hadamard Sketch has recently become popular in several applications throughout numerical  linear algebra because it satisfies the JL Lemma~\cite{boutsidis2013improved} and it can be computed using $\cO (m^2 \log \tau)$ operations
using the  \emph{Fast Walsh-Hadamard Transform} (FWHT)~\cite{fino1976unified} or the recent \emph{Trimmed Walsh-Hadamard Transform} (FWHT)~\cite{ailon2009fast}.
Let $\mH$ be the Hadamard matrix defined in~\eqref{eq:hadamard_matrix} of order $m$. Such methods use a butterfly structure, like the Cooley-Tukey FFT algorithm, to compute $\mH y$ in $\cO \left(m \log m\right)$ time for any $y \in \R^m$, but they require $m$ to be a power of $2$.

This last detail is often glossed over in theory, but we find in practice this poses a challenge when $m$ is large.
One can remove some rows of $\mA$ to meet this assumption, with the risk of losing vital information. Instead,
when $m$ is not a power of $2$, we need to pad  $\mA$ with zero rows until the augmented matrix has $m'$ rows, where $m'$ is a power of $2$. That is, we need $m' = 2^{\lceil \log_2 m \rceil}$ rows.  In the worst case scenario, when $m = 2^q +1$ for some $q\in \N^*$ then $m' = 2^{q+1}$, effectively doubling the number of rows.

\section{Convergence theory}
\label{sec:conv_theory}

Here we present theoretical convergence guarantees of \texttt{RidgeSketch} method~\eqref{eq:xupdate} for  sketch matrices $\mS \in \R^{m \times \tau}$ drawn from a fixed distribution $\cD$.
Later on, in~\Cref{sec:mom} we provide a new convergence theory for a new momentum variant of the sketch-and-project method. To contextualize our contribution, we will first present the previously known convergence theory of sketch-and-project method.

\subsection{Convergence of the iterates}

The sketch-and-project method~\eqref{eq:xupdate} enjoys a linear convergence in L2 given as follows.
\begin{theorem}[Convergence of the iterates, \cite{SDA, Gower2015}]
	\label{thm:iterates_conv}
	Let $w^*$ be a solution of~\eqref{eq:Axb} and let  $w^0 \in \range{\mA}$. Let $t\in \N$ and consider the iterates given in~\eqref{eq:xupdateW}. It follows that
	\begin{equation} \label{eq:L2conv}
		\E{\norm{w^t - w^*}_{\mA}^2} \leq (1- \rho)^t \norm{w^0 -w^*}_{\mA}^2 \enspace,
	\end{equation}
	where
	\begin{equation}
		\rho = \lambda_{\min}^+ (\mA^{1/2}\E{\mS (\mS^\top \mA \mS)^{\dagger}\mS}\mA^{1/2}) \enspace.
	\end{equation}
Consequently, we also have for the residual that
	\begin{equation} \label{eq:L2convres}
		\E{\norm{\mA w^t - b}^2} \leq (1- \rho)^t \lambda_{\max}(\mA)\norm{w^0 -w^*}_{\mA}^2 \enspace.
	\end{equation}
\end{theorem}
\begin{proof}
The proof of~\eqref{eq:L2conv} is given in Theorem 1.1 in~\cite{SDA}. The convergence of the residual in~\eqref{eq:L2convres} follows as a consequence since
\begin{equation}
	\label{eq:complexL2}
	\norm{\mA w^t - b}^2 = \norm{ w^t - w^*}_{\mA^2}^2 \leq \lambda_{\max}(\mA)  \norm{ w^t - w^*}_{\mA}^2 \enspace.
\end{equation}
\end{proof}

This linear convergence in L2 is often thought of as a gold standard for convergence of stochastic sequences, since it implies convergence in high probability and of all the moments of the sequence. Furthermore, the error decays at an exponential rate determined by $\rho.$ Yet the downside of~\eqref{eq:L2conv} and~\eqref{eq:L2convres} is that the rate of convergence $\rho$ can be very small. Next we present a sublinear rate of convergence in L2 that has an improved rate of convergence.

\subsection{Convergence of the residuals with step size $0 < \gamma < 1$}

The convergence proofs we present next rely on viewing the sketch-and-project method as an instance of \emph{SGD} (stochastic gradient descent)~\cite{Richtarik2017stochastic, SNR}.
To establish this SGD viewpoint,
we first reformulate the problem of solving \eqref{eq:Axb} as the following minimization problem
\begin{equation}
	\label{eq:obj_sgd_reformulation}
	\min_{w \in \R^d} f (w) \eqdef \frac{1}{2} \norm{ \mA w - b }_{\E{\mH_{\mS}}}^2 = \E{\frac{1}{2} \norm{ \mA w - b }_{\mH_{\mS}}^2} =: \E{f_{\mS} (w)} \enspace,
\end{equation}
where
\begin{equation}
	\label{eq:HS}
	\mH_{\mS} \eqdef \mS (\mS^\top \mA \mW^{-1} \mA \mS)^{\dagger} \mS^\top \enspace.
\end{equation}

Solving our linear systems is now equivalent to solving the stochastic minimization problem~\eqref{eq:obj_sgd_reformulation}, for which the classic method is SGD.  It just so happens that, SGD with a step size of $\gamma = 1$ is exactly the  Sketch-and-Project iteration~\eqref{eq:xupdateW}. Indeed, if we compute the gradients with respect to the $\norm{w}_{\mW} = \sqrt{\dotprod{\mW w, w}}$ norm, then SGD is given by
\begin{align}
	w^{k+1} &= w^k - \gamma\nabla_{\mW} f_{\mS} (w^k) \enspace, \label{eq:sgd_step}
\end{align}
where the gradient relative to the weighted inner product $\dotprod{., .}_{\mW}$ is given by
\begin{equation}
	\label{eq:grad}
	\nabla_{\mW} f_{\mS} (w^k) \; \eqdef \; \mW^{-1} \mA \mH_{\mS} \left(\mA w^k - b\right) \enspace,
\end{equation}
 where $\mS \sim \cD$ is sampled i.i.d at each iteration and
 $\gamma > 0$ is the step size. Since $f_{\mS}(w)$ is an unbiased estimate of $f(w)$ we have that
\begin{equation}
	\label{eq:gradunbiased}
	\E{\nabla_{\mW} f_{\mS} (w^k)} \; = \; \nabla_{\mW} f(w^k).
\end{equation}

Using the interpretation as SGD, we can provide the convergence of $\E{f(w^k)}$ to zero. We can then relate the convergence of $f(w^k)$ to the convergence of the residual using the following lemma.
\begin{lemma}
\begin{equation}
	\label{eq:suboptres}
	\lambda_{\min}\left(\E{\mH_{\mS}}\right) \norm{\mA w -b}^2 \; \leq \; 2 f(w) \enspace.
\end{equation}
\end{lemma}
\begin{proof}
This follows from
\begin{equation*}
	\lambda_{\min}\left(\E{\mH_{\mS}}\right) \norm{\mA w - b}^2
	\; = \; \norm{\mA w - b}_{\lambda_{\min}\left(\E{\mH_{\mS}}\right) \mI_m}^2
	\; \leq \; \norm{\mA w - b}_{\E{\mH_{\mS}}}^2
	\; = \; 2 f( w) \enspace.
\end{equation*}
\ 
\end{proof}

First we need the following property of the gradient taken from Lemma 3.1 in~\cite{Richtarik2017stochastic}.
\begin{lemma}[Gradient norm -- function identity] It follows that
\begin{equation}
	\label{eq:gradeqf}
	\norm{\nabla_{\mW} f_{\mS} (w)}_{\mW}^2 \; = \;  2 f_{\mS} (w) \enspace.
\end{equation}
\end{lemma}
\begin{proof}
	For completeness we give the proof. By straight forward computation we have that
	\begin{eqnarray}
	\norm{\nabla_{\mW} f_{\mS} (w)}_{\mW}^2 &\overset{\eqref{eq:grad}}{=}& \dotprod{\mW^{-1} \mA \mH_{\mS} \left(\mA w - b\right) , \mW^{-1} \mA \mH_{\mS} \left(\mA w - b\right)}_{\mW} \nonumber \\
	&=&\left(\mA w - b\right)^\top \mH_{\mS} \mA \mW^{-1} \mA \mH_{\mS} \left(\mA w - b\right) \nonumber \\
	&\overset{\eqref{eq:HS}}{=}& \left(\mA w - b\right)^\top \mH_{\mS} \left(\mA w - b\right) \nonumber \\
	&\overset{\eqref{eq:obj_sgd_reformulation}}{=}& 2 f_{\mS} (w) \enspace. \nonumber
	\end{eqnarray}
	where in the third equality we expanded $\mH_{\mS}$ and applied the identity $\mM^\dagger = \mM^\dagger \mM \mM^\dagger$ with $\mM = \mS^\top \mA \mW^{-1} \mA \mS$.
\end{proof}
Furthermore, the  functions $f_{\mS}$ are convex.
\begin{lemma} The function $f_{\mS}$ is a convex quadratic. Consequently
\begin{equation}
	\label{eq:convex}
	f_{\mS}(y) \geq f_{\mS}(x) + \dotprod{\nabla_{\mW} f_{\mS}(x), y - x}_{\mW} \enspace.
\end{equation}
\end{lemma}
\begin{proof}
Let $w \in \R^d$ and $\mS \sim \cD$, the gradient relative to the Euclidean inner product $\dotprod{., .}$ is
\begin{equation*}
	\nabla f_{\mS} (x) = \mA \mH_{\mS} \left(\mA x - b\right) \enspace,
\end{equation*}
and thus the hessian is
\begin{equation*}
	\nabla^2 f_{\mS} (x) = \mA \mH_{\mS} \mA \enspace,
\end{equation*}
which is a semi-definite positive matrix. This implies that $f_{\mS}$ is convex.
As a consequence~\eqref{eq:convex} holds since
\begin{align*}
	f_{\mS}(y) &\geq f_{\mS}(x) + \dotprod{\nabla f_{\mS}(x), y - x} \\
			   &= f_{\mS}(x) + \dotprod{\nabla_{\mW} f_{\mS}(x), y - x}_{\mW} \enspace,
\end{align*}
and given that $\dotprod{\nabla_{\mW} f_{\mS}(x), y - x}_{\mW} =\dotprod{\mW\mW^{-1}\nabla f_{\mS}(x), y - x} = \dotprod{\nabla f_{\mS}(x), y - x}$.
\end{proof}

Next we establish  the sublinear convergence of the \emph{average of the iterates} of Sketch-and-project method. This result is a direct consequence of Theorem 4.10 in~\cite{Richtarik2017stochastic} and has already been proven in Theorem 3 in~\cite{Loizou2020mom}. We present the complete statement and proof since it is a warm-up for our forthcoming results, and since the proof is substantially simpler than the one presented in~\cite{Loizou2020mom}.
\begin{theorem}[Convergence of the residuals]
	\label{thm:residual_conv}
	Let $\gamma \in (0, 1)$. Let $w^*$ be a solution of~\eqref{eq:Axb}, let $t\in \N^*$ and let $(w^k)_{0 \leq k \leq t}$ be the iterates given by~\eqref{eq:sgd_step}. It follows that
	\begin{equation}
		\label{eq:convsublinear}
	 	f \left(\overline{w}^t\right) \; \leq \; \frac{1}{t} \sum_{k=0}^t f \left(w^k\right) \; \leq \; \frac{1}{t} \frac{ \norm{w^0 - w^*}_{\mW}^2 }{2\gamma(1- \gamma ) } \enspace,
	\end{equation}
	where $\overline{w}^t \eqdef \frac{1}{t} \sum_{k=0}^{t} w^k.$ Furthermore
	\begin{equation}
		\label{eq:convressublinear}
		\norm{\mA \overline{w}^t - b}^2 \;\leq  \;\frac{1}{t}  \frac{1}{\lambda_{\min}\left(\E{\mH_{\mS}}\right) }  \frac{ \norm{w^0 - w^*}_{\mW}^2 }{2\gamma(1- \gamma ) } \enspace.
	\end{equation}
\end{theorem}
\begin{proof}
\begin{eqnarray*}
	\norm{w^{k+1} - w^*}_{\mW}^2 & =& \norm{w^k - w^*}_{\mW}^2 -2\gamma \dotprod{\nabla_{\mW} f_{\mS} \left(w^k\right), w^k - w^*}_{\mW} + \gamma^2 \norm{\nabla_{\mW} f_{\mS} \left(w^k\right)}_{\mW}^2 \\
	&\overset{\eqref{eq:gradeqf}}{=} & \norm{w^k - w^*}_{\mW}^2 -2\gamma \dotprod{\nabla_{\mW} f_{\mS} \left(w^k\right), w^k - w^*}_{\mW} +2 \gamma^2 f_{\mS} \left(w^k\right) \\
	&\overset{\eqref{eq:convex}}{\leq } &\norm{w^k - w^*}_{\mW}^2 -2\gamma(1- \gamma )  f_{\mS} \left(w^k\right) \enspace.
\end{eqnarray*}
Re-arranging, and dividing through by $\gamma(1- \gamma ) >0$ we have that
\begin{eqnarray*}
	f_{\mS} \left(w^k\right) & \leq & \frac{1}{2\gamma(1- \gamma ) } \left( \norm{w^k -w^*}_{\mW}^2 -\norm{w^{k+1} -w^*}_{\mW}^2\right) \enspace.
\end{eqnarray*}
Summing up over both sides for $k=0, \ldots, t$ and using telescopic cancellation we have that
\begin{eqnarray}
	\label{eq:sum_eqs}
	\sum_{k=0}^t f_{\mS} \left(w^k\right) & \leq & \frac{1}{2\gamma(1- \gamma ) } \left( \norm{w^0 -w^*}_{\mW}^2 -\norm{w^{t+1} -w^*}_{\mW}^2\right) \; \leq \;\frac{\norm{w^0 -w^*}_{\mW}^2 }{2\gamma(1- \gamma ) } \enspace.
\end{eqnarray}
By applying the Jensen's inequality to $f_{\mS}$, which is convex since it is a quadratic form,
\begin{equation}
	\label{eq:jensen}
	f_{\mS} \left(\frac{1}{t} \sum_{k=0}^t w^k\right) \leq \frac{1}{t} \sum_{k=0}^t f_{\mS} \left(w^k\right) \enspace,
\end{equation}
Now, by denoting $\overline{w}^t \eqdef \frac{1}{t} \sum_{k=0}^t w^k$, the convergence result~\eqref{eq:convsublinear} follows by dividing~\eqref{eq:sum_eqs} by $t$ and using~\eqref{eq:jensen}. Finally, the convergence of the residual in~\eqref{eq:convressublinear} follows from~\eqref{eq:suboptres}.
\end{proof}

The weakness of~\Cref{thm:residual_conv} is that it describes how the average of the iterates $\overline{w}^t$ converge instead of the last one $w^t$.
This type of convergence is problematic since $\overline{w}^t$ gives as much importance, a $1/t$ weight, to the initial point $w^0$ as to the last one $w^t$.
Since $w^0$ is often chosen arbitrarily, the average of the iterates can converge only as fast as $w^0$ is forgotten.
That is, the average cannot converge faster than $1/t$.
This is  apparent in experiments, where using averaging from the start results in a slow convergence.
In practice, averaging only the last few iterates works substantially better, but it is not supported in theory.
To resolve this issue, we will replace this \emph{equal} averaging with a \emph{weighted} average that gives more weight to recent iterates, and forgets the initial conditions exponentially fast.

\section{Momentum}
\label{sec:mom}

A common variant of SGD is to add momentum.
Since the sketch-and-project method can be interpreted as SGD~\eqref{eq:sgd_step}, we can add momentum. Let $\gamma_k \in [0,\;1 ]$ and $\beta_k  \in [0,\;1 ]$
be respectively the step size and the momentum parameter. The heavy ball formulation of momentum is given by
\begin{equation}
	\label{eq:sgd_stepmom}
	w^{k+1} = w^k -\gamma_k\nabla_{\mW} f_{\mS} \left(w^k\right) +\beta_k \left(w^k-w^{k-1}\right) \enspace.
\end{equation}
Note that we have now allowed for a step size $\gamma_k$ that is iteration dependent.

This same heavy ball formulation was considered in~\cite{Loizou2020mom}, where the authors propose a precise analysis and show no benefit using momentum.
But, all of their analysis assumes that the momentum parameter is constant. It turns out, that by allowing $\beta_k$ to be iteration dependent, we can do better.
But first, we need the \emph{iterative averaging viewpoint} of momentum.

\subsection{Iterate Averaging Viewpoint}

Recently, a new \emph{iterate averaging}  parametrization of the momentum method was proposed in~\cite{Sebbouh2020}.
This iterative averaging parametrization is given by
\begin{align}
	z^k & = z^{k-1} -\eta_k \nabla_{\mW} f_{\mS} (w^k) \label{eq:zupdateap}\\
	w^{k+1} & =\left(1 -  \frac{1}{\zeta_{k+1}+1}\right) w^k + \frac{1}{\zeta_{k+1}+1}  z^k \enspace, \label{eq:iterateavap}
\end{align}
where we have introduced two new parameter sequences $\eta_k$ and  $ \zeta_k$ that map back to the $\gamma_k$ and $\beta_k$ parameters via
\begin{equation}
	\label{eq:etalambdatogammbetagen}
	\gamma_k =\frac{\eta_k}{\zeta_{k+1}+1}\quad \mbox{and}\quad \beta_k = \frac{\zeta_k}{\zeta_{k+1}+1} \enspace.
\end{equation}

Next we show that~\eqref{eq:iterateavap} produces the same $w_k$ iterates as~\eqref{eq:sgd_stepmom}.

\begin{proposition}[Equivalent formulations]
	\label{prop:equivalence_mom}
	Let $t \in \N^*$. The steps given in~\eqref{eq:sgd_stepmom} and in~\eqref{eq:zupdateap}--\eqref{eq:iterateavap} generate the same iterates $(w_k)_{0 \leq k \leq t}$.
\end{proposition}
\begin{proof}
	First, let us expand the \emph{iterate averaging} parametrization
	\begin{align}
		w^{k+1} &\overset{\eqref{eq:iterateavap}}{=} \left(1 - \frac{1}{\zeta_{k+1}+1} \right) w^k + \frac{1}{\zeta_{k+1}+1} z^k \nonumber \\
		&\overset{\eqref{eq:zupdateap}}{=} w^k - \frac{1}{\zeta_{k+1}+1} w^k + \frac{1}{\zeta_{k+1}+1} \left(z^{k-1} -\eta \nabla_{\mW} f_{\mS} (w^k)\right) \nonumber \\
		&= w^k - \frac{\eta}{\zeta_{k+1}+1} \nabla_{\mW} f_{\mS} (w^k) + \frac{1}{\zeta_{k+1}+1} \left(z^{k-1} - w^k\right) \nonumber \\
		&= w^k - \gamma_k \nabla_{\mW} f_{\mS} (w^k) + \frac{1}{\zeta_{k+1}+1} \left(z^{k-1} - w^k\right) \enspace, \label{eq:expanding_iterateavgen}
	\end{align}
	where in the last step we used the left-hand side relation in~\eqref{eq:etalambdatogammbetagen}.
To conclude the proof, we need only show that
	\[\frac{1}{\zeta_{k+1}+1} \left(z^{k-1} - w^k\right)  = \beta_k (w_k-w_{k-1}) \enspace. \]
Indeed, this follows  $w^k = \left(1 - \frac{1}{\zeta_{k}+1} \right) w^{k-1} + \frac{1}{\zeta_{k}+1} z^{k-1}$, which can be rearranged in $z^{k-1} - w^k = \zeta_{k} \left(w^k - w^{k-1}\right)$, consequently
\[ \frac{1}{\zeta_{k+1}+1} \left(z^{k-1} - w^k\right) = \frac{\zeta_k}{\zeta_{k+1}+1} \left(w^{k} - w^{k-1}\right)
\overset{\eqref{eq:etalambdatogammbetagen}}{=} \beta_k (w_k-w_{k-1}) \enspace, \]
which together with~\eqref{eq:expanding_iterateavgen} concludes the proof.
\end{proof}

Next we show how to leverage the iterative averaging viewpoint~\eqref{eq:zupdateap}--\eqref{eq:iterateavap} to prove the convergence of the last iterates of sketch-and-project with momentum~\eqref{eq:sgd_stepmom}.

\subsection{Convergence theorem}
\label{sec:momconv}

\begin{theorem}
	\label{thm:momap}
	Consider the iterates~\eqref{eq:zupdateap}--\eqref{eq:iterateavap}  with $w^{0} =w^{-1}$. Let $\eta_k $ be a sequence of parameters with $0<\eta_k < 1$. If
	\begin{equation}
		\zeta_0 = 0 \quad \mbox{and}\quad \zeta_k = \frac{1}{\eta_k} \sum_{t=0}^{k-1} \eta_t (1- \eta_t) \enspace, \quad \mbox{for all }k \geq 1 \enspace. \label{eq:zeta_kaptheo}
	\end{equation}
	then
	\begin{eqnarray}
		\label{eq:momconv}
		\norm{\mA w^k -b}^2 & \leq  & \frac{1}{\lambda_{\min}\left(\E{\mH_{\mS}}\right)} \frac{ \norm{w^{0} -w^*}_{\mW}^2  }{ \sum_{t=0}^{k} \eta_t (1- \eta_t) } \enspace.
	\end{eqnarray}
\end{theorem}

\begin{proof}
This proof is based on Theorem 3.1 in~\cite{Sebbouh2020} for SGD applied to convex and smooth functions.
Consider the Lyapunov function
\begin{equation}
	L_k \; = \;\E{\norm{z^k-w^*}_{\mW}^2} + 2\eta_k\zeta_k \E{f(w^k)} \enspace.
\end{equation}
First note that
\begin{align*}
	\norm{z^k-w^*}_{\mW}^2 &\!\!\overset{\eqref{eq:zupdateap}}{=}\!\! \norm{z^{k-1} -\eta_k \nabla_{\mW} f_{\mS} (w^k) -w^*}_{\mW}^2  \\
	&= \norm{z^{k-1} -w^*}_{\mW}^2 -2\eta_k \dotprod{z^{k-1} -w^*,  \nabla_{\mW} f_{\mS} (w^k) }_{\mW} + \eta_k^2 \norm{\nabla_{\mW} f_{\mS} (w^k) }_{\mW}^2 \\
	&\!\!\overset{\eqref{eq:iterateavap}}{=}\! \norm{z^{k-1} -w^*}_{\mW}^2 \!-\!2\eta_k \dotprod{w^k -w^* +\zeta_k(w^k-w^{k-1}, \nabla_{\mW} f_{\mS} (w^k) }_{\mW} \!+\! \eta_k^2 \norm{\nabla_{\mW} f_{\mS} (w^k) }_{\mW}^2 \\
	&= \norm{z^{k-1} -w^*}_{\mW}^2 -2\eta_k \dotprod{w^k -w^*, \nabla_{\mW} f_{\mS} (w^k) }_{\mW} \!-\! 2\eta_k \zeta_k\dotprod{ w^k-w^{k-1}, \nabla_{\mW} f_{\mS} (w^k) }_{\mW}\\
	&\quad + \eta_k^2 \norm{\nabla_{\mW} f_{\mS} (w^k) }_{\mW}^2 \\
	&\overset{\eqref{eq:gradeqf}}{=} \norm{z^{k-1} -w^*}_{\mW}^2  + 2\eta_k^2 f_{\mS} (w^k) - 2\eta_k \dotprod{w^k -w^* ,  \nabla_{\mW} f_{\mS} (w^k) }_{\mW} \\
	&\quad -2\eta_k \zeta_k\dotprod{ w^k-w^{k-1}, \nabla_{\mW} f_{\mS} (w^k) }_{\mW} \enspace.
\end{align*}
The above holds to equality. Now we introduce the first inequality by calling upon~\eqref{eq:convex} so that
\begin{eqnarray}
	\norm{z^k-w^*}_{\mW}^2  & \overset{\eqref{eq:convex}}{\leq}    &
	\norm{z^{k-1} -w^*}_{\mW}^2  + 2\eta_k^2 f_{\mS} (w^k)  -2\eta_k f_{\mS} (w^k)    -2\eta_k \zeta_k (f_{\mS} \left(w^k\right) - f_{\mS}(w^{k-1})) \nonumber \\
	& =&  \norm{z^{k-1} -w^*}_{\mW}^2  - 2\eta_k(1+\zeta_k-\eta_k) f_{\mS} (w^k)     +2\eta_k \zeta_k f_{\mS}(w^{k-1}) \enspace. \label{eq:ai8ha88ja4ap}
\end{eqnarray}
The restriction on the parameters in~\eqref{eq:zeta_kaptheo}  was designed so that
\begin{equation}\label{eq:soj9so9sj4ap}
	\eta_k(1+\zeta_k-\eta_k) \; = \; \eta_{k+1}\zeta_{k+1} \enspace.
\end{equation}
Indeed, from~\eqref{eq:zeta_kaptheo} we have that
\begin{align*}
	\zeta_{k+1} &= \frac{1}{\eta_{k+1}} \sum_{t=0}^{k} \eta_t (1- \eta_t) \;= \frac{\eta_k}{\eta_{k+1}}\frac{1}{\eta_k} \sum_{t=0}^{k} \eta_t (1- \eta_t) \;=\frac{\eta_k}{\eta_{k+1}} \left( \zeta_k +  1- \eta_k \right) \enspace.
\end{align*}
Using~\eqref{eq:soj9so9sj4ap} in~\eqref{eq:ai8ha88ja4ap} and taking expectation gives
\begin{eqnarray}
	\E{\norm{z^k-w^*}_{\mW}^2}+ 2\eta_{k+1}\zeta_{k+1} \E{f(w^k)} & \leq  & \E{ \norm{z^{k-1} -w^*}_{\mW}^2}    +2\eta_k \zeta_k \E{f(w^{k-1})} \enspace.
\end{eqnarray}
Summing up both sides from $k=0,\ldots, t$ and using telescopic cancellation gives
\begin{eqnarray}
	\E{\norm{z^t-w^*}_{\mW}^2}+ 2\eta_{t+1}\zeta_{t+1} \E{f(w^t)} & \leq  & \E{ \norm{z^{-1} -w^*}_{\mW}^2}    +2\eta_0 \zeta_0 \E{f(w^{-1})} \enspace.
\end{eqnarray}
Using that $z^{-1} = w^0$ from~\eqref{eq:iterateavap}, $\zeta_0 =0$ and re-arranging gives
\begin{eqnarray}
	\E{\norm{z^t-w^*}_{\mW}^2}+ 2\eta_{t+1}\zeta_{t+1} \E{f(w^t)} & \leq  & \E{ \norm{w^{0} -w^*}_{\mW}^2} \enspace.
\end{eqnarray}
Substituting the definition of $\zeta_{t+1}$ from~\eqref{eq:zeta_kaptheo} gives
\begin{eqnarray}
	2\E{f(w^t)} & \leq  & \frac{1}{ \eta_{t+1}\zeta_{t+1} } \E{ \norm{w^{0} -w^*}_{\mW}^2}   =  \frac{1}{ \sum_{k=0}^{t} \eta_k (1- \eta_k) } \E{ \norm{w^{0} -w^*}_{\mW}^2} \enspace.
\end{eqnarray}
The final step is a result of using~\eqref{eq:suboptres} to lower bound  $\E{f(w^t)}$.
\end{proof}

Theorem~\ref{thm:momap} shows that the convergences of the iterates $w^t$ depend on a sequence of parameters $\eta_k$. Next we give a corollary that shows that by simply choosing $\eta_k \equiv \eta  <1$ the iterates $w^t$ enjoy a fast sublinear convergence.

\begin{corollary}
	\label{cor:momconve}
Consider the setting of Theorem~\ref{thm:momap}. Let $\eta_k \equiv \eta <1$ and thus
\begin{equation}
	\zeta_0  = 0 \quad \mbox{and}\quad \zeta_k = k(1- \eta) \enspace, \quad \mbox{for all }k \geq 1 \enspace, \label{eq:zeta_formula_cst_eta}
\end{equation}
and
\begin{equation}
	\label{eq:cst_eta_lambda_to_gamma_beta_formula}
	\gamma_k =\frac{\eta}{(k+1)(1- \eta)+1}\quad \mbox{and}\quad \beta_k =1 - \frac{ 2-\eta}{(k+1)(1- \eta)+1} \enspace,
\end{equation}
then
\begin{eqnarray}
	\norm{\mA w^k -b}^2 & \leq  & \frac{1}{\lambda_{\min}\left(\E{\mH_{\mS}}\right)}  \frac{1}{  k} \frac{ \norm{w^{0} -w^*}_{\mW}^2}{ \eta (1- \eta) } \enspace.
\end{eqnarray}
Consequently,  for $\eta = \frac{1}{2}$ and for a given tolerance $\epsilon >0$, the iteration complexity of minimizing the residual is given by
	\begin{equation}
		\label{eq:complexsublinear}
		t \; \geq \; \frac{4\norm{w^0 - w^*}_{\mW}^2}{\lambda_{\min}\left(\E{\mH_{\mS}}\right)} \frac{1}{\epsilon} \enspace,
	\end{equation}
	to reach a desired precision
	\begin{equation*}
	\E{\norm{\mA w^t - b}^2}< \epsilon \enspace.
	\end{equation*}
\end{corollary}

Corollary~\ref{cor:momconve} shows that the \emph{last iterate} of sketch-and-project with momentum converges  at the same rate as the average of the iterates of sketch-and-project (see Theorem~\ref{thm:residual_conv}). This is the first tangible theoretical advantage of using momentum in this setting.

In~\Cref{sec:specialconv}, through a more specialized setting for ridge regression, we show that the complexity given in~\eqref{eq:complexsublinear} can be even tighter than the previously known linear convergence in Theorem~\ref{thm:iterates_conv}.
But first, we present a practical pseudo-code implementation of sketch-and-project with momentum.

\subsection{Implementation}

In Algorithm~\ref{alg:momentum_algo} we have the pseudo-code of Sketch-and-Project with momentum~\eqref{eq:sgd_stepmom} for solving ridge regression~\eqref{eq:ridge}  where $\mW = \mA.$

\begin{algorithm}[!h]
	\begin{algorithmic}[1]
		\State \textbf{Parameters:} distribution over random $m \times \tau$ matrices in $\cD$, tolerance $\epsilon > 0$,
		\Statex momentum parameters $\eta_k, \zeta_k \in [0, 1]$
		\State Set $w^{-1} = w^0 = 0 \in \R^m$ \Comment weights initialization
		\State Set $r^{-1} = r^0 = \mA w^0 - b = - b \in \R^m$ \Comment residual initialization
		\State $k = 0$
		\While {$\norm{r^k}_2 / \norm{r^0}_2 \leq \epsilon$}
			\State $\gamma_k =\frac{\eta_k}{1 + \zeta_{k+1}}$  and   $\beta_k = \frac{\zeta_k}{1 + \zeta_{k+1}}$
			\State Sample an independent copy $\mS_k \sim \cD$
			\State $r_{\mS}^k = \mS_k^\top r^k$ \Comment compute sketched residual
			\State $\delta_k =$ \texttt{least\_norm\_solution} $\left( \mS_k^\top \mA \mS_k, r_{\mS}^k \right)$ \Comment solve sketched system
			\State $w^{k+1} = (1 + \beta_k) w^k - \beta_k w^{k-1} - \gamma_k \mS_k \delta_k$ \Comment update the iterates
			\State $r^{k+1} = (1 + \beta_k) r^k - \beta_k r^{k-1} - \gamma_k \mA \mS_k \delta_k$ \Comment update the residual
			\State $k = k + 1$
		\EndWhile
		\State \textbf{Output:} $w^{t}$ \Comment return weights vector
	\end{algorithmic}
	\caption{The Sketch-and-Project method with Momentum}
	\label{alg:momentum_algo}
\end{algorithm}

The residual $r^k \eqdef \norm{\mA w^k -b}^2$ is required  for computing the update like in Algorithm~\ref{alg:SP}. Fortunately we can efficiently compute the residual at step $k+1$  by storing the residuals at two steps $k-1$ and $k$ since
\begin{eqnarray}
	r^{k+1} &=& \mA w^{k+1} - b \nonumber \\
	&\overset{\eqref{eq:sgd_stepmom}}{=}& \mA ((1 + \beta_k) w^k - \beta_k w^{k-1} - \gamma_k 	 \nabla_{\mW} f_{\mS} (w^k)) - b \nonumber \\
		&\overset{\eqref{eq:grad}}{=}& \mA ((1 + \beta_k) w^k - \beta_k w^{k-1} - \gamma_k \mS_k \delta_k) - b \nonumber \\
	&\overset{\eqref{eq:grad}}{=} &(1 + \beta_k) r^k - \beta_k r^{k-1} - \gamma_k \mA \mS_k \delta_k \enspace, \label{eq:res_update_mom}
\end{eqnarray}
where
\[\delta_k = \left( \mS_k^\top \mA \mS_k \right)^{\dagger} \mS_k^\top r^k = \texttt{least\_norm\_solution} \left( \mS_k^\top \mA \mS_k, r_{\mS}^k \right) \enspace,\]
and where we used that since $\mW = \mA$ we have that
\[ \nabla_{\mW} f_{\mS} (w^k) \;\overset{\eqref{eq:grad}}{=}\; \mW^{-1} \mA  \mS (\mS^\top \mA \mW^{-1} \mA \mS)^{\dagger} \mS^\top \left(\mA w^k - b\right)\; =\; \mS_k \left( \mS_k^\top \mA \mS_k \right)^{\dagger} \mS_k^\top r^k \enspace.\]
Thus, for the momentum version of our algorithm we can keep the residual update at the cost of $O(m \tau)$ by just storing the residual  $r^{k-1}$  at the previous time step.

\section{Specialized convergence theory for single column sketches}
\label{sec:specialconv}

Here we take a closer look at the rates of convergence given by Theorem~\ref{thm:iterates_conv} and Corollary~\ref{cor:momconve} by considering a specialized setting of ridge regression ($\mW = \mA$) and single column sketches.
That is, in this section we use a discrete distribution for $\mS$ given by
\begin{equation}
	\label{eq:mSdisc}
	\Prob{\mS = s_i} \; = \; p_i, \quad \mbox{for } i=1, \ldots, m \enspace,
\end{equation}
where $s_1, \ldots, s_m \in \R^{m}$ are a fixed collection of vectors and  $\sum_{i=1}^q p_i =1$.

To better understand the complexity~\eqref{eq:complexsublinear} and~\eqref{eq:complexL2} , we first need to find a lower bound for $\lambda_{\min}\left(\E{\mH_{\mS}}\right)$ where
\begin{equation}
	\label{eq:EHSWA}
	\E{\mH_{\mS}} \;= \;\E{\mS (\mS^\top \mA \mS)^{\dagger} \mS^\top} \enspace.
\end{equation}
By using a special choice for the probabilities in~\eqref{eq:mSdisc} , we are able to give a convenient lower bound for $\E{\mH_{\mS}}$ in the following lemma.

\begin{lemma}
	\label{lem:singlecolbnd}
	Let $\mS$  have a discrete distribution according to
	\begin{equation}
		\label{eq:mSdisccol}
		\Prob{\mS = s_i} \; = \; \frac{s_i^\top \mA s_i}{\sum_{j=1}^q s_j^\top \mA s_j} \enspace, \quad \mbox{for }i=1,\ldots, q \enspace,
	\end{equation}
	where $s_1, \ldots, s_q\in \R^m$ are unit column vectors.  Let $\mF \eqdef [s_1, \ldots, s_q] \in \R^{m \times q}.$ It follows that
	\begin{equation}
		\label{eq:lambdaEHbndcol}
		\lambda_{\min}(\E{\mH_{\mS}}) \; = \; \frac{ \lambda_{\min} (\mF\mF^\top)}{\sum_{j=1}^m s_j^\top \mA s_j} \enspace.
	\end{equation}
\end{lemma}
\begin{proof}
This \emph{convenient} probability distribution~\eqref{eq:mSdisccol} was already considered in~\cite{Gower2015} in Section 5.2.
But there in, the authors used these probabilities to study a different spectral quantity, thus for completion we adapt their proof to our setting.
First note that
	\begin{eqnarray*}
	\E{\mH_{\mS}} & \overset{\eqref{eq:EHSWA}}{=} & \sum_{i=1}^q \frac{p_i}{s_i^\top \mA s_i } s_i s_i^\top
	\;  \overset{\eqref{eq:mSdisccol}}{=}  \; \frac{\sum_{i=1}^q s_i s_i^\top}{\sum_{j=1}^q s_j^\top \mA s_j} \enspace.
	\end{eqnarray*}
	Consequently, the smallest eigenvalue is given by
	\[		\lambda_{\min}(\E{\mH_{\mS}}) \; = \; 		\lambda_{\min}\left(\frac{\sum_{i=1}^q s_i s_i^\top}{\sum_{j=1}^q s_j^\top \mA s_j} \right) \; = \; \frac{ \lambda_{\min} (\mF\mF^\top)}{\sum_{j=1}^m s_j^\top \mA s_j} \enspace.\]
\end{proof}

\subsection{Comparing the complexity of CD with and without Momentum}

Here we compare the fast sublinear convergence of  \texttt{RidgeSketch} with momentum given in Theorem~\eqref{thm:momap}  to the linear convergence of \texttt{RidgeSketch} without momentum given in  Theorem~\ref{thm:iterates_conv}.
Though linear convergence is generally preferred, we will show here that our new sublinear rate of convergence can be faster.
To illustrate this, we will focus on the special case of Coordinate Descent (CD).

The CD method is the result of applying \texttt{RidgeSketch} when the sketching matrices are unit coordinate vectors.
That is, when $s_i = e_i \in \R^{m}$ is the $i$-th column of the identity matrix and
\begin{equation}
	\label{eq:mSdisccolCD}
	\Prob{s_i = e_i} \; = \; \frac{\mA_{ii}}{\trace{ \mA }} \enspace, \quad \mbox{for }i=1,\ldots, m  \enspace.
\end{equation}
This particular nonuniform sampling in~\eqref{eq:mSdisccolCD} was first given in~\cite{Leventhal2010}. With this sketch, the \texttt{RidgeSketch} method with momentum~\eqref{eq:sgd_stepmom} becomes CD with momentum which is given by
\begin{equation} \label{eq:CDmom}
	w^{k+1} = w^k -  \gamma_k \frac{\mA_{i:} w^k-b_i}{\mA_{ii}} e_i + \beta_k (w^k-w^{k-1}) \enspace.
\end{equation}

The CD method with constant momentum $\beta_k \equiv \beta$ is known to converge linearly~\cite{Leventhal2010,Loizou2020mom}. But the linear convergence of CD with momentum is always slower than CD without momentum, see Theorem 1 in~\cite{Loizou2020mom}.
Here we present the first convergence rate of CD with momentum that can be faster than CD without momentum. But first, we need the following corollary.

\begin{corollary}
	\label{cor:complexCD}
	Let $\epsilon >0$.
	If we set the momentum parameters $\gamma_k$ and $\beta_k$ according to~\eqref{eq:cst_eta_lambda_to_gamma_beta_formula} with $\eta \equiv 0.5$ then the iterates of CD with momentum~\eqref{eq:CDmom} satisfy
	\begin{equation}
		\label{eq:complexsublinearCD}
		t \geq 	4\frac{\trace{\mA}}{\epsilon}  \; \implies \; \frac{\E{\norm{\mA w^t - b}^2} }{\norm{w^0 -w^*}_{\mA}^2} < \epsilon  \enspace. 
	\end{equation}

	Alternatively, if we use no momentum (setting $\gamma_k=1$ and $\beta_k =0$) then the iterates~\eqref{eq:CDmom} satisfy
	\begin{equation}
		\label{eq:linearcomplexCD}
		t \; \geq \;	 \frac{\trace{\mA}}{\lambda_{\min}(\mA)} \log\left(\frac{\lambda_{\max}(\mA)}{\epsilon}\right)  \; \implies \; \frac{\E{\norm{\mA w^t - b}^2} }{\norm{w^0 -w^*}_{\mA}^2} < \epsilon \enspace .
	\end{equation}
\end{corollary}
\begin{proof}
Consider the sampling given by~\eqref{eq:mSdisccolCD}.
From~\eqref{eq:lambdaEHbndcol}, since $\mF = \mI_m$, we have that
\begin{equation}
	\label{eq:lambdaEHbndCD}
	\lambda_{\min}(\E{\mH_{\mS}}) \; \geq \; \frac{ 1}{\trace{ \mA }} \enspace.
\end{equation}
Consequently, using~\eqref{eq:lambdaEHbndCD} together with the complexity bound~\eqref{eq:complexsublinear} and setting $\mW = \mA$ we have~\eqref{eq:complexsublinearCD}.

The linear complexity~\eqref{eq:linearcomplexCD} follows from a special case of Theorem~\ref{thm:iterates_conv}. Indeed, by using~\eqref{eq:lambdaEHbndCD} we have that
\[\rho = \frac{\lambda_{\min}(\mA)}{\trace{ \mA }} \enspace.\]
The resulting complexity~\eqref{eq:linearcomplexCD} follows by standard manipulations of logarithm.
\end{proof}

This linear convergence~\eqref{eq:linearcomplexCD} is generally preferred because of the resulting logarithmic dependency of $\epsilon.$ But, as we show next, the sublinear complexity given in~\eqref{eq:complexsublinearCD} can be tighter when $\epsilon$ is not too small.

\begin{corollary}[Domain of superiority of the sublinear over the linear convergence]
	\label{cor:complexCDcompare}
	Consider the setting of Corollary~\ref{cor:complexCDcompare}. Let $\hat{\epsilon} \; \eqdef \; \frac{\epsilon}{\lambda_{\max}(\mA)}$ be the \emph{scaled precision} and let use denote $\kappa  \eqdef \frac{\lambda_{\max}(\mA)}{\lambda_{\min}(\mA)}$ the \emph{condition number}.
	If
	\begin{equation}
		\label{eq:complexCDcompare}
		\hat{\epsilon}(1 -\hat{\epsilon}) \; \geq \; \frac{4}{\kappa} \enspace,
	\end{equation}
	then the complexity bound of momentum~\eqref{eq:complexsublinearCD} is tighter than the bound in~\eqref{eq:linearcomplexCD}.

Furthermore, if $\kappa \geq 16,$ the solutions to~\eqref{eq:complexCDcompare} in $\hat{\epsilon}$ are given by
\begin{equation}
	\label{eq:bounds_scaled_eps}
	\frac{1}{2} - \frac{1}{2}\sqrt{1-\frac{16}{\kappa}} \leq  \hat{\epsilon} \leq \frac{1}{2} + \frac{1}{2}\sqrt{1-\frac{16}{\kappa}} \enspace.
\end{equation}
\end{corollary}
\begin{proof}
The complexity bound in~\eqref{eq:complexsublinearCD} is tighter than the bound in~\eqref{eq:linearcomplexCD} if
\[ \frac{\trace{\mA}}{\lambda_{\min}(\mA)} \log\left(  \frac{\lambda_{\max}(\mA)}{\epsilon}\right)  \;  \geq  \; 4\frac{\trace{\mA}}{\epsilon} \enspace. \]
Substituting $ \hat{\epsilon} \; \eqdef \; \frac{\epsilon}{\lambda_{\max}(\mA)}$ and re-arranging the above gives
\begin{equation}
	\label{eq:tempmaniuaeopja}
	\hat{\epsilon}\log\left(\frac{1}{ \hat{\epsilon}}\right) \; \geq \; \frac{4}{\kappa} \enspace.
\end{equation}
To further bound the above we use the following standard logarithm bound
$$ \frac{1}{1-x} \log\left(\frac{1}{x}\right)  \geq 1 \enspace, \quad \forall x \in (0,\; 1)\enspace,$$
	which after manipulations gives
$$x \log\left(\frac{1}{x}\right)  \geq x(1 -x) \enspace.$$
Assuming that $\hat{\epsilon} = \epsilon / \lambda_{\max}(\mA) \leq 1$ and using this  bound with $x =\hat{\epsilon}$ we have that if~\eqref{eq:complexCDcompare} holds then
$$  \hat{\epsilon}\log\left(\frac{1}{\hat{\epsilon}}\right) \; \geq \; \hat{\epsilon}(1 -\hat{\epsilon}) \; \geq \; \frac{4}{\kappa} \enspace.$$
Thus~\eqref{eq:tempmaniuaeopja} holds.
\end{proof}

Using Corollary~\ref{cor:complexCDcompare} we can deduce several regimes where the sublinear momentum bound in~\eqref{eq:complexsublinearCD} is tighter than the bound in~\eqref{eq:linearcomplexCD}. As illustrated in~\Cref{fig:superiority_sublinear_over_linear}, when the condition number is moderate to large or the scaled precision is moderate, the sublinear bound~\eqref{eq:complexsublinearCD} is often tighter.
On the other hand, when $\hat{\epsilon}$ is very small, then linear rates such as~\eqref{eq:linearcomplexCD} are generally preferred.

The regime where sketch-and-projecting methods are interesting is when $\epsilon$ is moderate, and $\mA$ has large dimensions. Furthermore, in the large dimensional setting, the condition number of $\mA$ can also be very large, thus~\eqref{eq:complexCDcompare} is likely to hold.

\section{RidgeSketch momentum experiments}
\label{sec:mom_exp}

Our first experiment explores the efficiency of the momentum version of our \texttt{RidgeSketch} method.
We then compare the different sketches described in~\Cref{sec:sketch_matrices}.
Finally, we prove the efficiency of our method on large scale real datasets and show it is competitive with CG and direct solvers.

\paragraph{Datasets.} In what follows, we test our algorithms on the datasets with different number of data samples $n$ and number of features $d$: \texttt{California Housing} ($n=20,640$, $d=8$), \texttt{Boston} ($n=506$, $d=13$), \texttt{RCV1} ($n=804,414$, $d=47,236$) fetched from \texttt{sklearn}\footnote{\url{https://scikit-learn.org/stable/modules/generated/sklearn.datasets}}~\cite{scikit-learn} and \texttt{Year Prediction MSD} ($n=515,345$, $d=90$) from the UCI repository\footnote{\url{https://archive.ics.uci.edu/ml/datasets/yearpredictionmsd}}.
We converted \texttt{RCV1} into a regression task by transforming multi-class labels into integers.

\subsection{Experiment 1: Comparison of different momentum settings}
\label{sec:exp_9.2}

In this section, we compare the sketch-and-project method with momentum for three settings: our new iteration dependent parameters given by~\eqref{eq:etalambdatogammbetagen}, the constant $\beta$ setting proposed in~\cite{Loizou2020mom} ($\gamma_k = 1$, $\beta_k = 0.5$) and no momentum at all ($\gamma_k = 1$, $\beta_k = 0$).
We report iteration plots since the sketch-and-project methods with or without momentum have almost the same iteration cost.
Indeed in either case, this cost is dominated by sketching $\mA$, computing the residual $r^k$ and solving the sketched system~\eqref{eq:sketched_system}.
We report error areas (1st and 3rd quartiles) computed over $10$ runs each.

\paragraph{Increasing momentum.}
As suggested by~\Cref{thm:momap}, our momentum Algorithm~\ref{alg:momentum_algo} requires choosing the sequence $\eta_k$, after which $\gamma_k$ and $\beta_k$ are set using~\eqref{eq:zeta_kaptheo} and~\eqref{eq:etalambdatogammbetagen}.
After running several benchmarks tests, we identified the following \emph{theoretical} rule for setting $\eta_k$
\begin{equation}
	\label{eq:parameterswitchtheory}
	\eta_k \; = \;
	\begin{cases}
	0.995 & \mbox{if } \beta_k <0.5 \\
	1 &\mbox{if } \beta_k \geq 0.5
	\end{cases} \enspace .
\end{equation}
We call this parameter choice increasing momentum as it allows $\beta_k$ to increase from $0$ to $0.5$ while the step size $\gamma_k$ decreases from $1$ to $0.5$, as showed in~\Cref{fig:eta_cst_then_1}.
\begin{figure}
    \centering
    \begin{subfigure}{0.35\textwidth}
		\includegraphics[width=\textwidth]{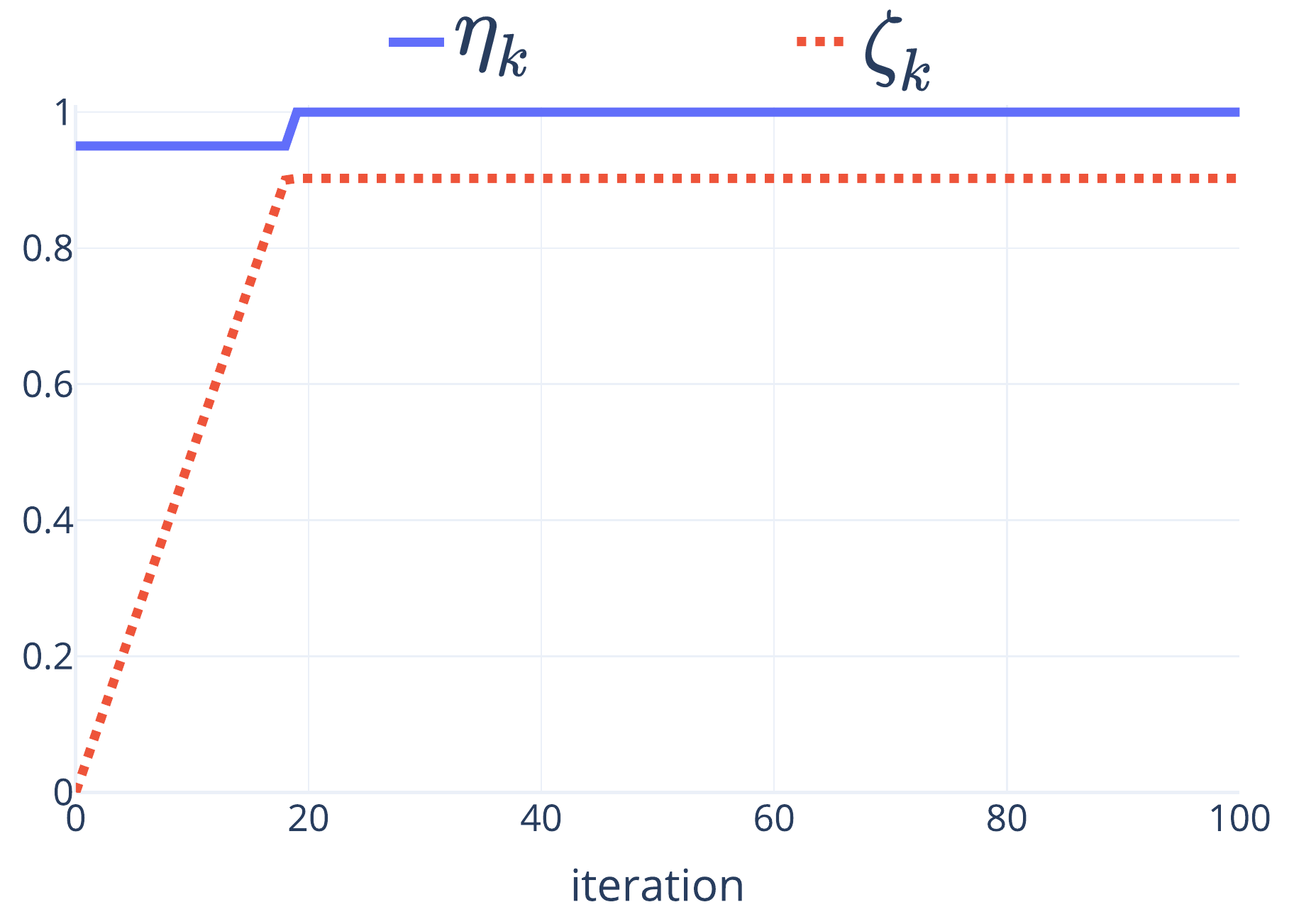}
		\caption{Parameters $\eta_k$ and $\zeta_k$}
    \end{subfigure}
    \hspace{1cm}
    \begin{subfigure}{0.35\textwidth}
		\includegraphics[width=\textwidth]{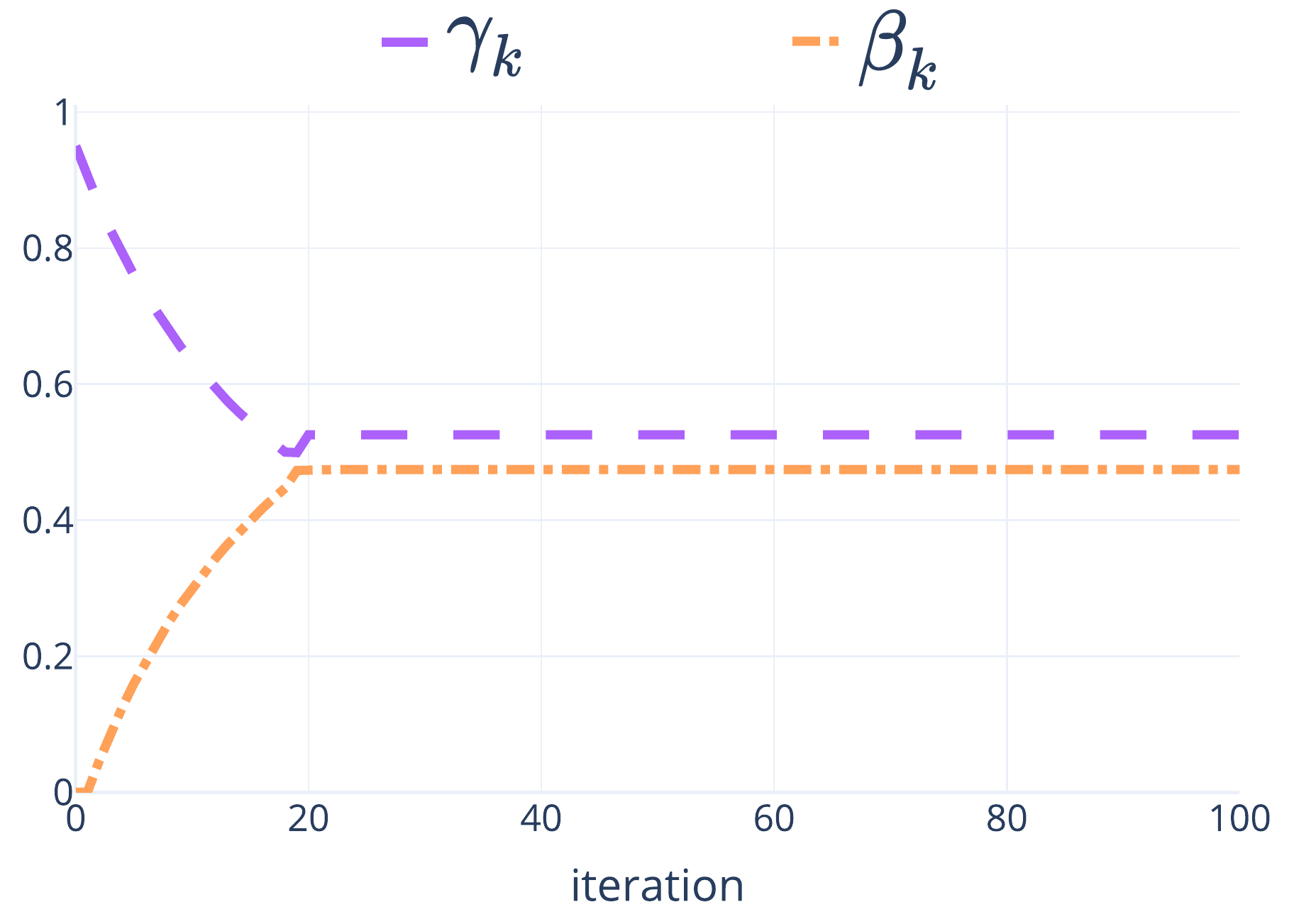}
        \caption{Step size and momentum parameter}
    \end{subfigure}
	\caption{Momentum parameters for the theoretical rule: $\eta_k \equiv 0.5$ while $\beta_k < 0.5$, then $\eta_k \equiv 1$.}
    \label{fig:eta_cst_then_1}
\end{figure}
Yet, we found in our experiments that this theoretical setting~\eqref{eq:parameterswitchtheory} closely matches the version without momentum of~\Cref{alg:SP}.
We suppose that the gain of the increasing momentum is lost by an excessively rapid drop of the step size to $0$.
This is why we introduce the following \emph{heuristic} setting that keeps the step size to $1$ and still uses the theoretical setting for momentum when $\eta_k \equiv \eta$ given by~\eqref{eq:cst_eta_lambda_to_gamma_beta_formula}, that is
\begin{equation}
	\label{eq:parametersheurististic}
	\gamma_k \equiv 1 \quad \mbox{and}\quad \beta_k =1 - \frac{ 2-\eta}{(k+1)(1- \eta)+1} \enspace.
\end{equation}
We tested it for increasing momentum on the \texttt{Boston} and \texttt{RCV1} datasets with different sketches and sketch sizes, see~\Cref{fig:effect_reg_on_mom_boston,fig:effect_sketch_size_mom_cali_k}.
We found that this heuristic setting~\eqref{eq:parametersheurististic} had the \emph{best of both worlds}, in that in the first iterations, when $\gamma_k \approx 1$ and $\beta_k \approx 0$, it benefits from the fast initial decrease of the no-momentum version.
Then, in later iterations, it exploits the fast asymptotic convergence of momentum since $\beta_k \approx 0.5$.

\paragraph{Regularization test.} Using the heuristic setting~\eqref{eq:parametersheurististic}, we tested the impact of using a small, medium and larger regularization parameter $\lambda$ on the performance of momentum, see~\Cref{fig:effect_reg_on_mom_boston}.
In this figure, we can see  that constant momentum is less effective as $\lambda$ increases, and the no momentum variant is more effective when $\lambda$ is small. Moreover, we observe the robustness of our heuristic increasing momentum since it performs well for all regularizers.

\begin{figure}
    \centering
	\begin{subfigure}{0.3\textwidth}
		\includegraphics[width=\textwidth]{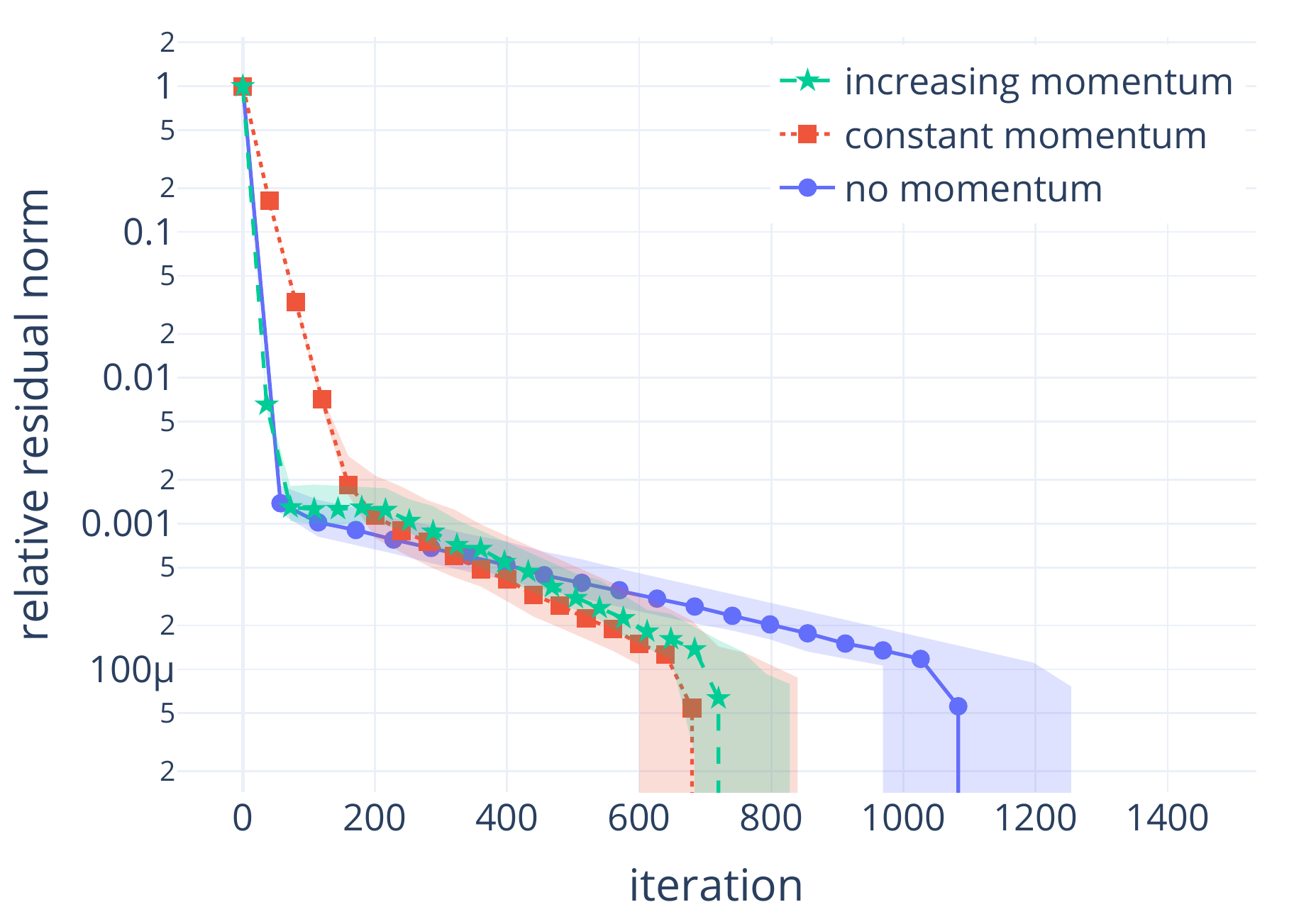}
        \caption{$\lambda = 0$}
    \end{subfigure}
    \begin{subfigure}{0.3\textwidth}
		\includegraphics[width=\textwidth]{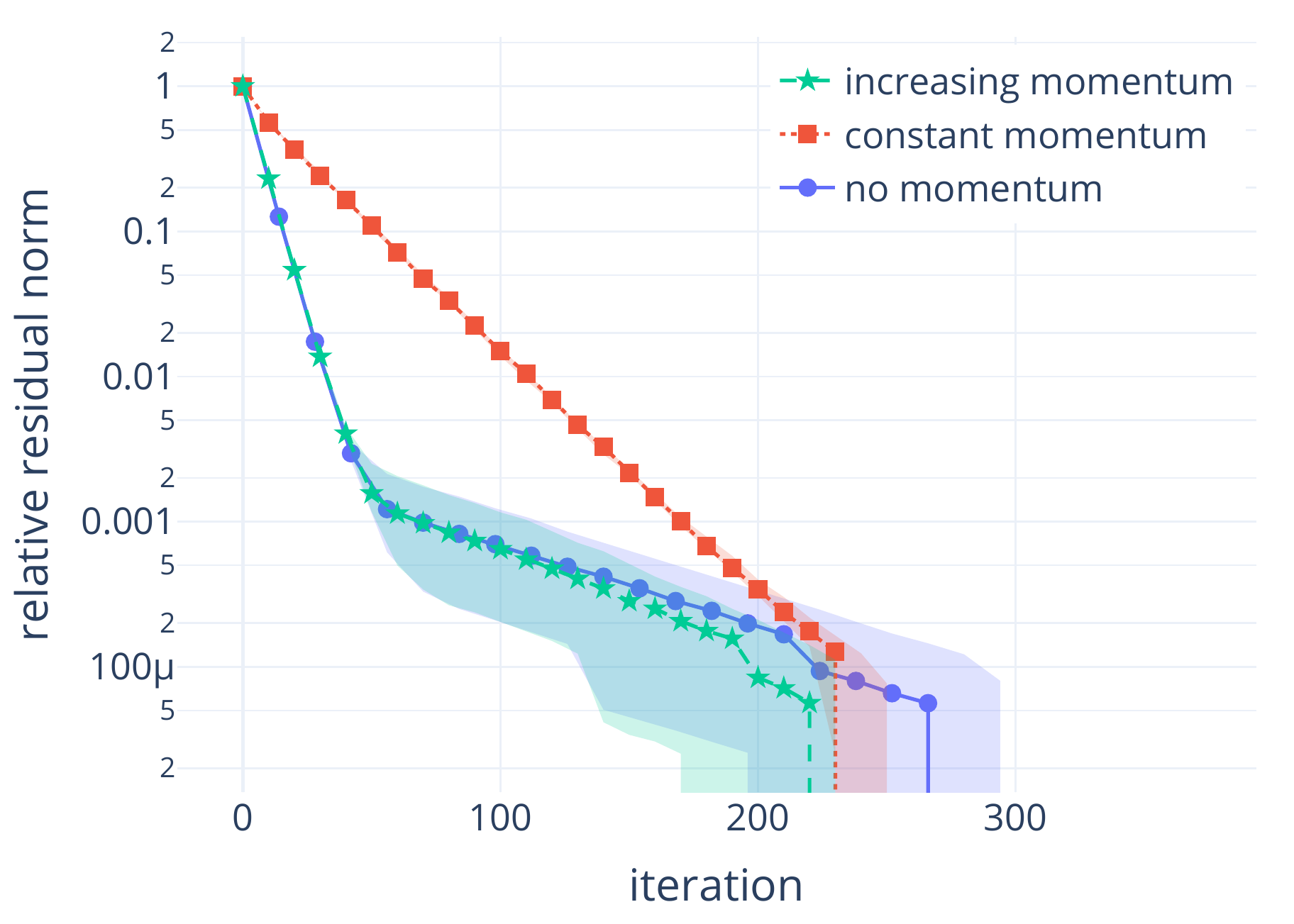}
        \caption{$\lambda = 0.7$}
    \end{subfigure}
	\begin{subfigure}{0.3\textwidth}
		\includegraphics[width=\textwidth]{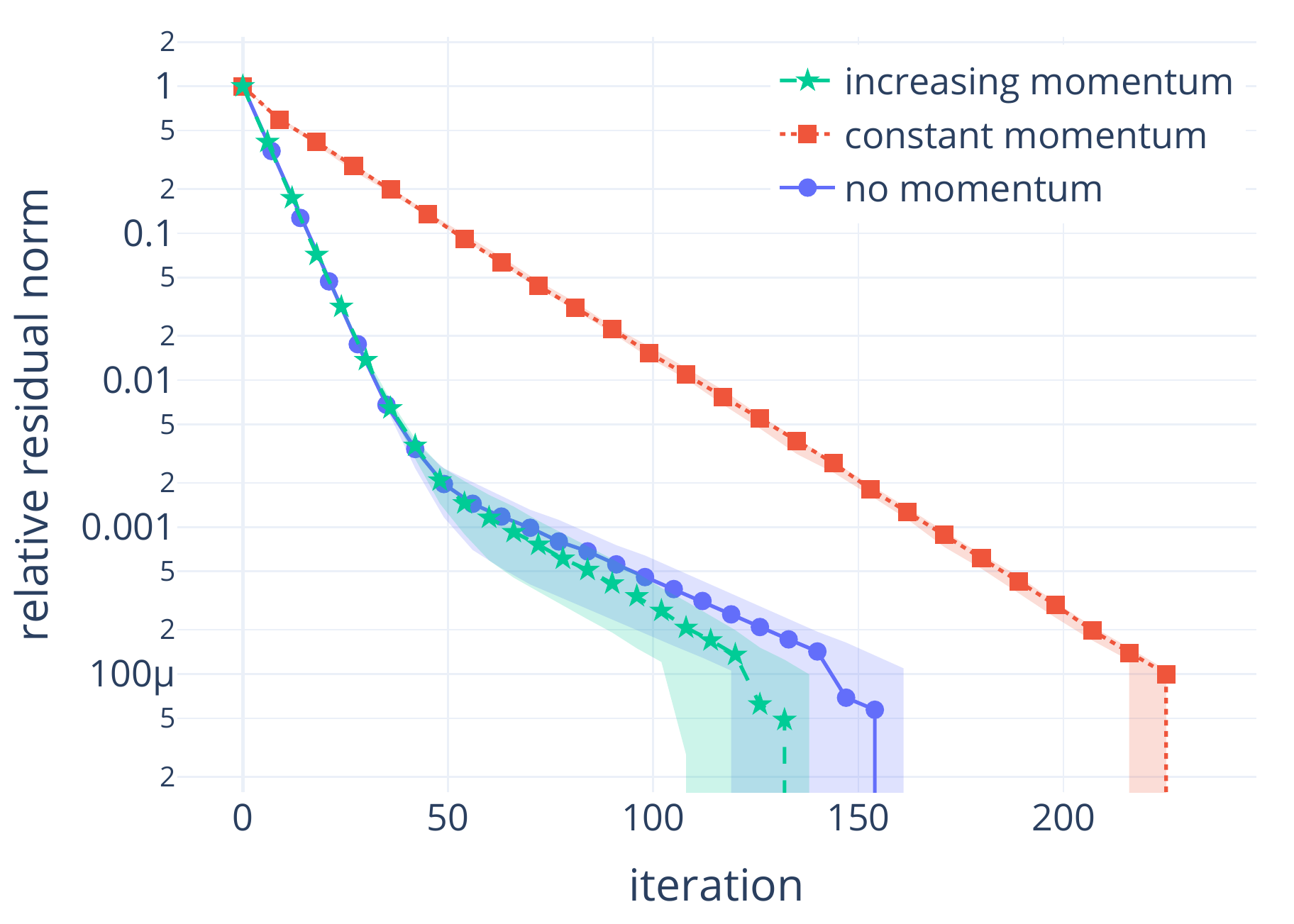}
		\caption{$\lambda = 70000$}
    \end{subfigure}
    \caption{Effect of regularizer $\lambda$ on increasing momentum (green), constant momentum (red) and no-momentum (blue) (\texttt{Boston} dataset with kernel, $m=506$, $\tau=\floor{m/4}=126$, \texttt{Count sketch}).}
    \label{fig:effect_reg_on_mom_boston}
\end{figure}

\paragraph{Sketch size.}
We tested different values of the sketch size, namely $\tau = 10\%, 50\%$ and $90\%$ of $m$, and reported the run time to reach a tolerance of $10^{-4}$ for each method.
In~\Cref{fig:effect_sketch_size_mom_cali_k}, we observe that constant momentum is very affected by the sketch size and is always the slowest method.
For intermediate sketch sizes, like $\tau = m/2$, our increasing momentum competes with no-momentum.
We also see that $\tau$ should not be set too small nor too large.
Indeed, larger sketch sizes lead to better estimates of the initial system~\eqref{eq:Axb} by~\eqref{eq:sketchedAxb}.
But if the sketch size is too large, solving the sketched system~\eqref{eq:sketched_system} becomes very slow.

\begin{figure}
    \centering
	\includegraphics[width=.45\textwidth]{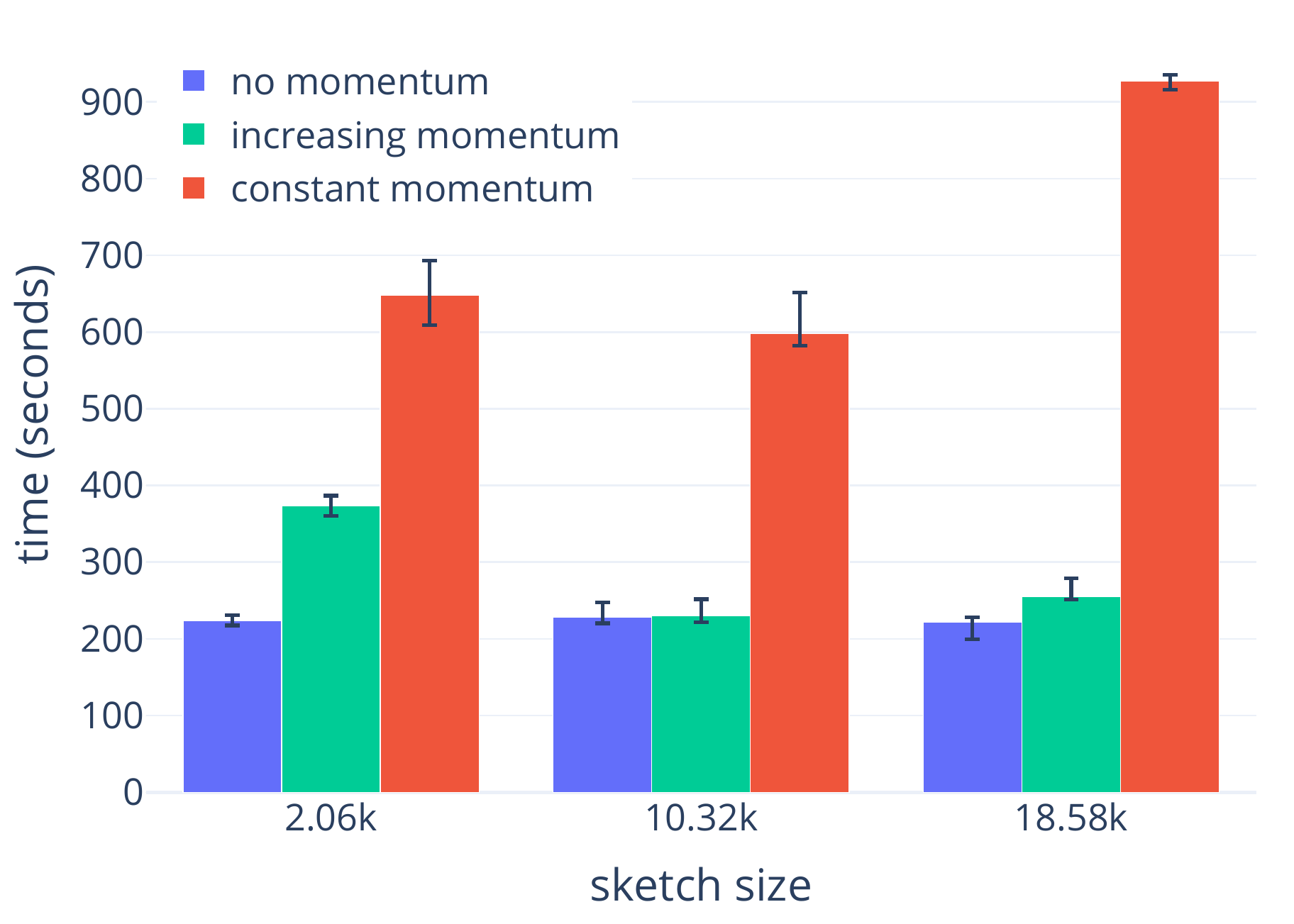}
	\caption{Effect of sketch size on increasing momentum (green), constant momentum (red) and no-momentum (blue) (\texttt{California Housing} with kernel, $m=20,640$, $\lambda = 10^{-6}$, \texttt{Subsample sketch}).}
	\label{fig:effect_sketch_size_mom_cali_k}
\end{figure}

\paragraph{Conclusions.}
We highlighted that momentum sketch-and-project is more efficient for small regularizers $\lambda$ as opposed to the vanilla method.
Also, we showed that the run time decreases then increases as a function of the sketch size $\tau$. Thus $\tau$ should be set to an intermediate value, \eg $\tau=m^{\frac{2}{3}}$, so that the cost of solving the sketched system~\eqref{eq:sketched_system} is manageable.
Finally, the main conclusion of this experiment is the overall robustness (across values of $\lambda$ and $\tau)$ and faster convergence of our heuristic increasing momentum setting.

\subsection{Experiment 2: Comparison of different types of sketches}
\label{sec:exp_9.3}

In this experiment we compare the performance of different sketching methods presented in~\Cref{sec:sketch_matrices} when using our heuristic increasing momentum setting~\eqref{eq:parametersheurististic}.
In~\Cref{fig:sketches_comparison_boston_kernel}, we monitor both the number of iterations and the time taken, since different sketching methods take different amounts of time per iteration.
We see in this figure that there is a clear ranking between sketch methods in terms of run time:
\begin{enumerate}
	\item \texttt{Subsample} is the most efficient on dense data (see also~\Cref{fig:exp_cali_kernel})
	\item \texttt{Count} and \texttt{SubCount} are the most efficient on sparse data (see~\Cref{fig:exp_rcv1}) and have very similar performance
	\item \texttt{Gaussian} is slow because of the cost of  dense matrix-matrix multiplications
	\item \texttt{Hadamard} is extremely slow because of the size of the padded matrix and of the preprocessing time it requires
\end{enumerate}

\paragraph{Conclusions.} For dense datasets, the \texttt{Subsample sketch} is the fastest because it only requires slicing operations, which are very well optimized (especially for NumPy arrays).
For sparse problems, the \texttt{Count sketch} is to be preferred since it \emph{densifies} just enough sketched matrices to extract information out of $\mA$.
We find that computing \texttt{Gaussian} and \texttt{Hadamard sketch} is very time demanding.
Furthermore, the cost associated to the padding step in \texttt{Hadamard sketch} is detrimental, especially for large $m$, which often makes it the slowest method.

\begin{figure}
    \centering
    \begin{subfigure}{0.4\textwidth}
		\includegraphics[width=\textwidth]{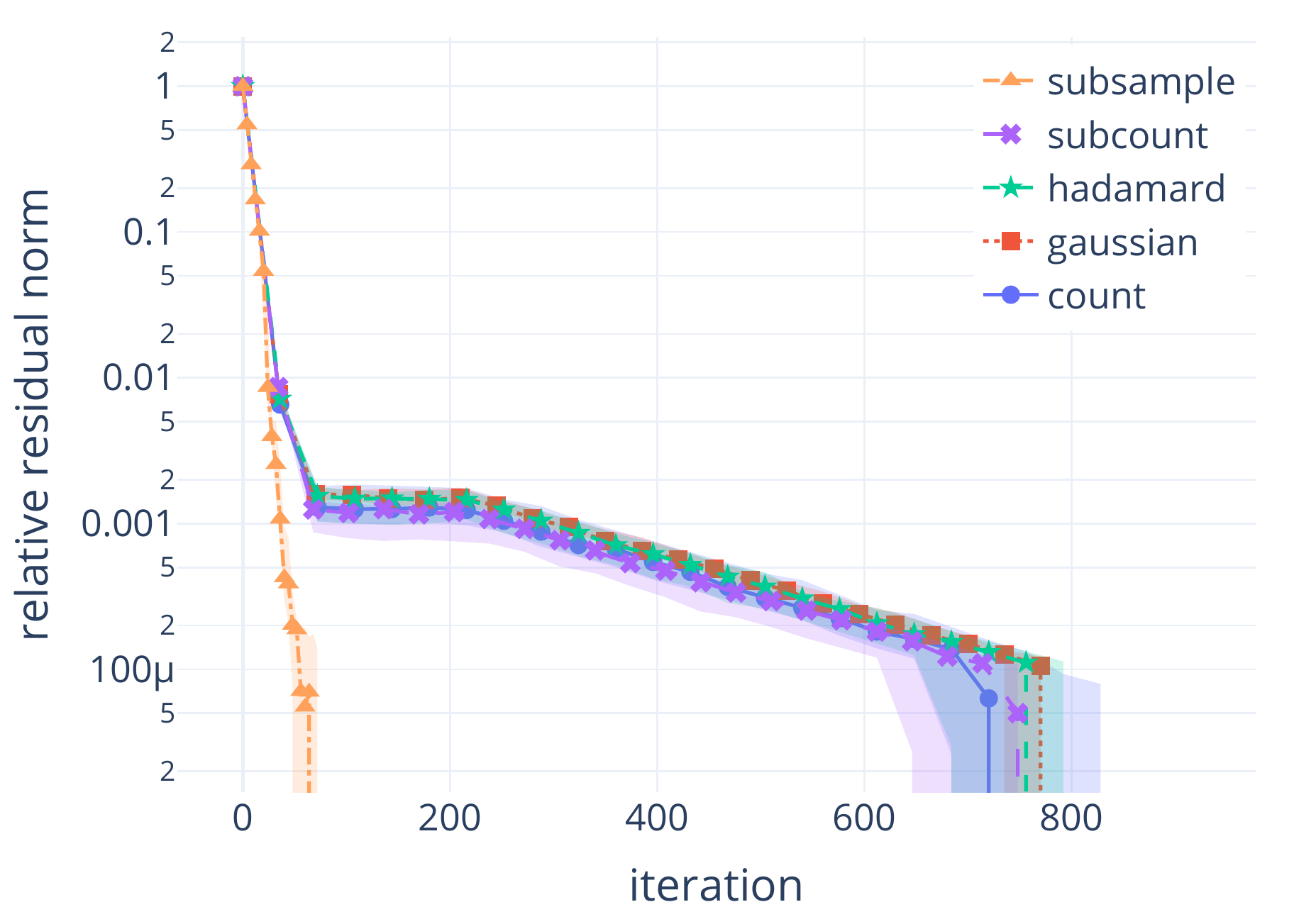}
		\caption{Iterations}
    \end{subfigure}
	\hspace{0.02\textwidth}
    \begin{subfigure}{0.4\textwidth}
		\includegraphics[width=\textwidth]{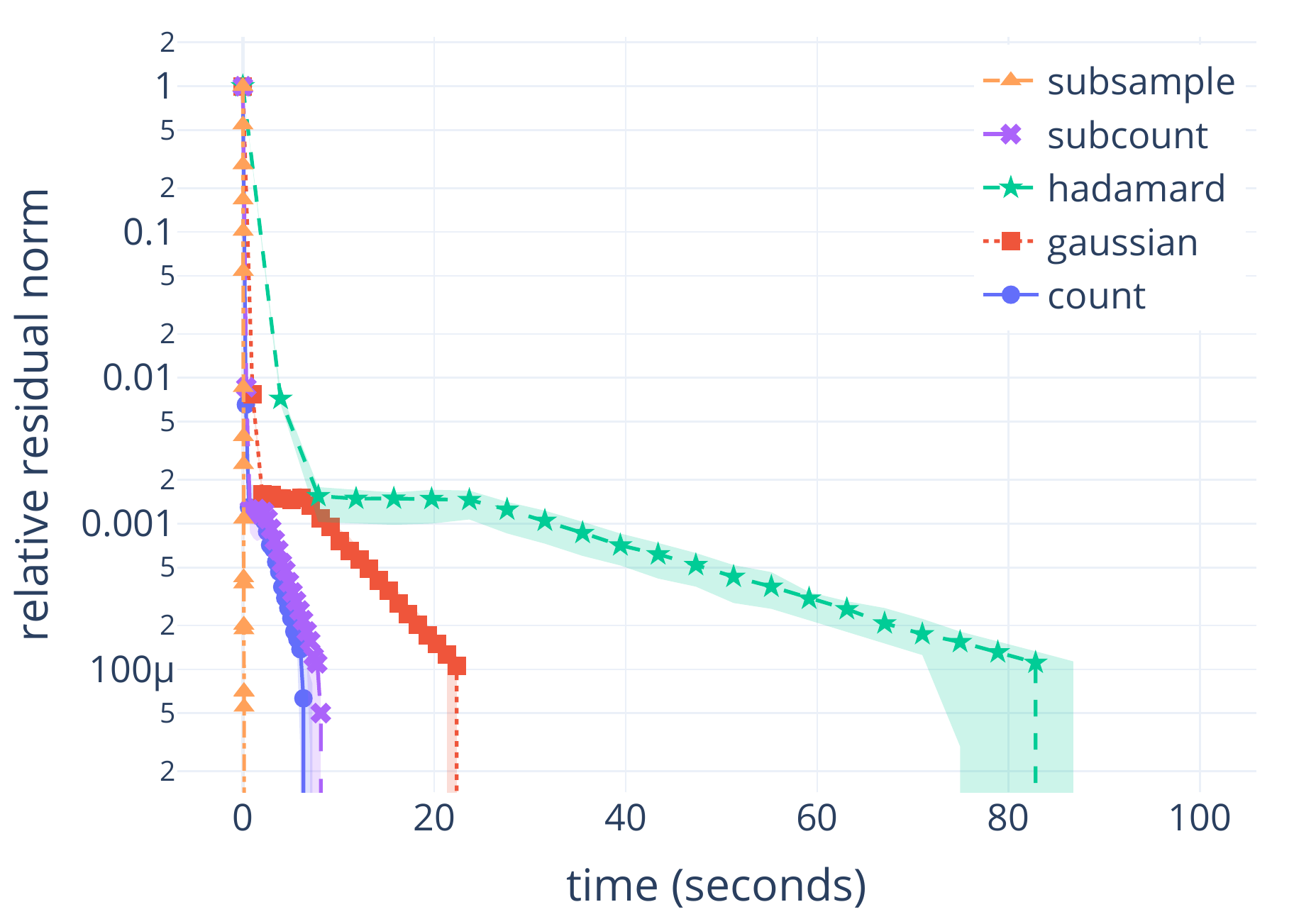}
        \caption{Time}
    \end{subfigure}
	\caption{Comparison of sketch methods for heuristic increasing momentum on kernel ridge regression problem applied to the \texttt{Boston} dataset ($\lambda=10^{-6}$, $m=506$ and $\tau= m/4 = 126$).}
    \label{fig:sketches_comparison_boston_kernel}
\end{figure}

\subsection{Experiment 3: Comparison against direct solver and conjugate gradients}

We now compare \texttt{RidgeSketch} with our heuristic increasing momentum setting~\eqref{eq:parametersheurististic}
with the two best sketches, \texttt{Subsample} and \texttt{Count sketch}, against a direct solver and \emph{Conjugate Gradients} (CG)~\cite{Hestenes1952}. The direct solver we used was LAPACK's \texttt{gesv} routine~\cite{LAPACK} for solving positive definite linear systems.
Here we tested our code on
\begin{itemize}
	\item A kernel ridge regression problem~\eqref{eq:kernelridge} on the  dataset \texttt{California Housing}  ($m = n = 20,640$).
	\item A large and sparse dataset: \texttt{RCV1}  $(m=d=47,236)$, with only $0.16\%$ of non-zeros.
\end{itemize}

\begin{figure}
    \centering
    \begin{subfigure}{0.4\textwidth}
		\includegraphics[width=\textwidth]{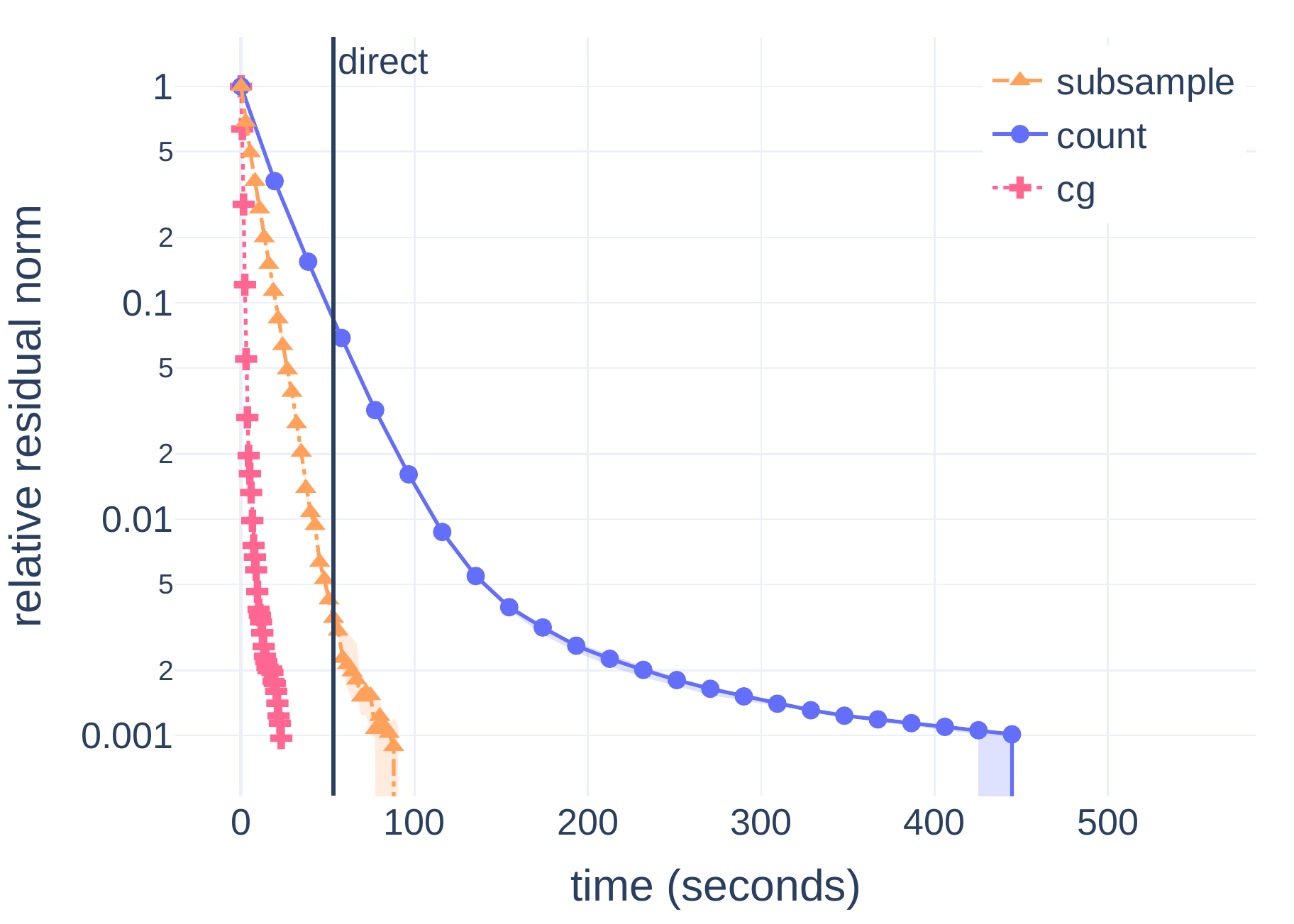}
		\caption{Kernel ridge regression applied to \texttt{California Housing} ($\tau=m/4 = 5,160$).}
		\label{fig:exp_cali_kernel}
    \end{subfigure}
	\hspace{0.02\textwidth}
    \begin{subfigure}{0.4\textwidth}
		\includegraphics[width=\textwidth]{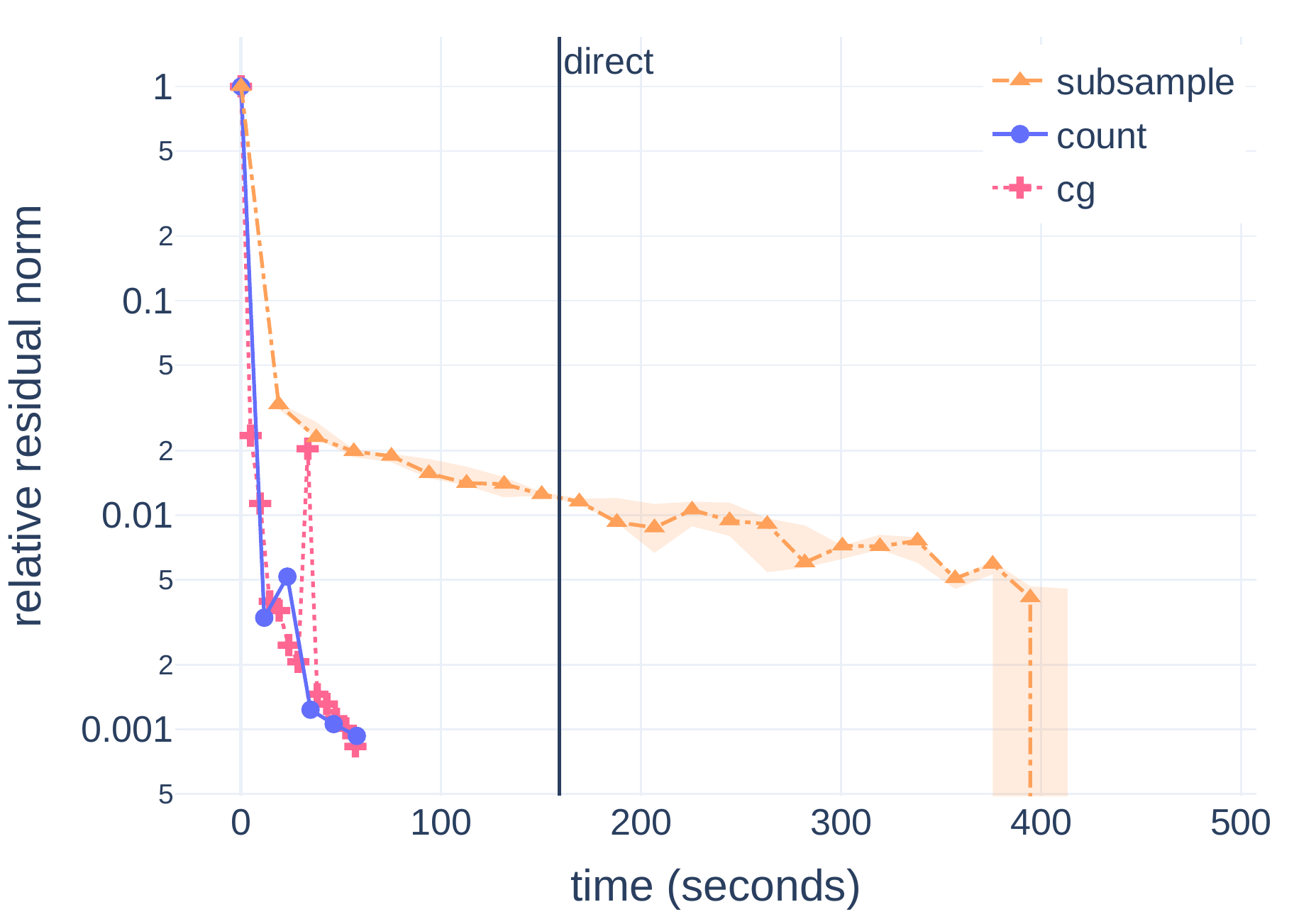}
        \caption{Ridge regression applied to \texttt{RCV1}\\($\tau=\ceil{m^{2/3}} = 1,307$).}
		\label{fig:exp_rcv1}
    \end{subfigure}
	\caption{Time comparison of best \texttt{RidgeSketch} methods (heuristic increasing momentum) against direct and CG solvers ($\lambda=10^{-9}$).}
    \label{fig:exp_4_comparison_cg_large_scale}
\end{figure}

\Cref{fig:exp_4_comparison_cg_large_scale} highlights a clear benefit of iterative methods like CG and sketch-and-project for solving large scale ridge problems. Moreover, this experiment on \texttt{RCV1} shows that in the large scale sparse setting, \texttt{Count sketch} is competitive  as compared to CG.

\section{Acceleration}
\label{sec:accel}

Recently, it was shown that the convergence rate of sketch-and-project can be improved by using \emph{acceleration}~\cite{TuVWGJR17,acellqN2018}.
See Algorithm~\ref{alg:accelSP} for our pseudo-code of the accelerated sketch-and-project method.
In~\cite{acellqN2018} it was shown that, by using specific parameter settings, the accelerated method enjoys a linear convergence with a rate that can be an order of magnitude better than rate given in Theorem~\ref{thm:iterates_conv}.

Despite this strong theoretical advantage of the accelerated method, it is not clear if this translates into a practical advantage because 1) the additional overhead costs of the method may outweigh the benefits of the improved iteration complexity and 2) the accelerated method relies on knowing beforehand spectral properties of the matrix $\mA$ that are expensive to compute.
Here we show that 1) can be remedied by a careful implementation and 2) is indeed a fundamental issue that prevented us from developing a practical method.

\subsection{From theory to practical implementation of acceleration}
\label{sec:accel_implementation}

\paragraph{Additional overhead: Pseudo-code and efficient implementation.}
The accelerated version in  Algorithm~\ref{alg:accelSP} has  three $(w^k,z^k,v^k)$ sequences of iterates.
The bottleneck costs of Algorithm~\ref{alg:accelSP} are the same as the standard sketch-and-project method in Algorithm~\ref{alg:SP}, which are the sketching operations on line~\ref{ln:stochgrad}.
Indeed, the only additional computations in Algorithm~\ref{alg:accelSP} as compared to Algorithm~\ref{alg:SP} are lines~\ref{ln:zupdate} and~\ref{ln:vupdate} which cost $O(m)$.
The other additional overhead is how to monitor the residual so as to know when to stop the algorithm.
We found that for the residual to be efficiently maintained and updated, we had to monitor three residual vectors $r_z^k \eqdef \mA z^k - b,$ $r_w^k \eqdef  \mA w^k - b $ and $r_v \eqdef  \mA v^k - b$.
These residual vectors can be updated efficiently since
\[r^{k+1}_w = \mA w^{k+1} - b = \mA(z^k - g^k) - b = r_z^k - \mA g^k = r_z^k - \mA \mS_k \delta^k.\]
Since we have already pre-computed $\mA \mS_k$ and $\delta^k$, the additional cost is $O(m \tau)$.
Furthermore from lines~\ref{ln:zupdate} and~\ref{ln:vupdate} we have that
\begin{align*}
	r^{k+1}_z &= \alpha \mA v^k + (1-\alpha) \mA w^k -b = \alpha r_v^k + (1-\alpha) r_w^k \enspace,
\end{align*}

\begin{algorithm}[ht]
	\begin{algorithmic}[1]

		\State \textbf{Parameters:} distribution over random $m \times \tau$ matrices in $\cD$, tolerance $\epsilon > 0$,
		\Statex acceleration parameters $\mu \in [0,\,1]$, $\nu \in \left[1, \, \tfrac{1}{\mu}\right]$

		\State Set $w^0 = v^0 = 0 \in \R^m$ \Comment weights initialization
		\State Set $r_z^0=r_v^0=r_w^0 = \mA w^0 - b \in \R^m$ \Comment residual initialization

		\State Set $\beta = 1 - \sqrt{\frac{\mu}{\nu}}$
		\State Set $\gamma = \sqrt{\frac{1}{\mu \nu}}$
		\State Set $\alpha = \frac{1}{1+\sqrt{\frac{\nu}{\mu }}}$
		\State $k = 0$
		\While {$\norm{r^k_v}_2 / \norm{r^0_v}_2 \leq \epsilon$}
			\State Sample  $\mS_k\sim \cD$ i.i.d
			\State Compute and store $\mA \mS_k$
			\State $\delta^k =$ \texttt{least\_norm\_solution} $\left( \mS_k^\top \mA \mS_k,\mS_k^\top  r_z^k \right)$ \Comment solve sketched system
			\State $g^k = \mS_k \delta^k$ \label{ln:stochgrad}
			\State $\mA g^k = (\mA \mS_k) \delta^k$
			\State $z^k = \alpha v^k + (1-\alpha) w^k$ \label{ln:zupdate} \Comment update the iterates
			\State $w^{k+1} = z^k -g^k$ \label{ln:xupdate}
			\State $v^{k+1} = \beta v^k +(1-\beta)z^k - \gamma g^k$ \label{ln:vupdate}
			\State $r^{k}_z = \alpha r_v^k + (1-\alpha) r_w^k $  \Comment update the residuals
			\State $r^{k+1}_w =  r_z^k - \mA g^k$
			\State $r^{k+1}_v = \beta r_v^k+ (1-\beta)r_z^k - \gamma \mA g^k$
			\State $k = k + 1$
		\EndWhile
		\State {\bf Output:} $w^{t}$ \Comment return weights vector
	\end{algorithmic}
	\caption{The Sketch-and-Project method with Acceleration}
	\label{alg:accelSP}
\end{algorithm}

and
\begin{align*}
	r_v^{k+1} &= \beta \mA v^k +(1-\beta)\mA z^k -b - \gamma \mA g^k = \beta r_v^k+ (1-\beta)r_z^k - \gamma \mA \mS_k \delta^k \enspace .
\end{align*}
Thus the residuals $r_v^k$ and $r_v^k$ can be updated at an additional $O(m)$ cost to perform the above vector additions and scalar multiplications.

\paragraph{Setting the acceleration parameters with spectral properties.}
The main issue with the accelerated version is that it introduces two new hyperparameters $\mu$ and $\nu$ which have to be estimated. In theory~\cite{acellqN2018}, by setting  these two parameters  according to
\begin{equation}
	\label{eq:mu+nu}
	\mu \eqdef \inf_{x \in \Range{\mA^\top}} \tfrac{\dotprod{\E{\mZ}x,x}}{\dotprod{x,x}}
	\qquad \mbox{and} \qquad
	\nu \eqdef \sup_{x \in \Range{\mA^\top}} \tfrac{\dotprod{\E{\mZ\E{\mZ}^\dagger \mZ}x,x}}{\dotprod{\E{\mZ}x,x}} \enspace.
\end{equation}
where
\begin{equation}
	\label{eq:Z}
	\mZ \eqdef \mA^\top\mS^\top(\mS^\top \mA\mS)^{\dagger}\mS^\top \mA \enspace,
\end{equation}
we can guarantee an accelerated rate of convergence. The issue is that the theory in~\cite{acellqN2018} requires that these parameters be set exactly using~\eqref{eq:mu+nu} and
computing~\eqref{eq:mu+nu} is more costly then solving the original linear system! So this leads us to the following practical question.

\begin{tcolorbox}
Is there a {\bf rule of thumb setting} for the acceleration parameters $\mu$ and $\nu$ such that the accelerated sketch-and-project method is consistently better than the standard sketch-and-project method?
\end{tcolorbox}

In~\cite{TuVWGJR17} the authors propose some settings for $\nu$ and $\mu$ when the  sketch size is large.  But there is currently no practical rule for setting these parameters in general.
In theory, we know that
\[ 0 \;\leq \; \mu \; \leq \;  \frac{1}{\nu} \; \leq  \; 1 \enspace,\]
as proven in Lemma 2 in~\cite{acellqN2018}.
Furthermore the extreme case where $\mu = \nu =1$ corresponds  to the standard sketch-and-project method, as can be seen by induction on Algorithm~\ref{alg:accelSP}  since $z^k = v^k = w^k $ for all iterations, and $w^k$'s are thus equivalent to the $w^k$'s in Algorithm~\ref{alg:SP}. We now look at some other extreme cases to better understand these parameters.

\paragraph{Single row sampling.}
For this special case of subsample sketchs with $\tau =1$, that is $\mS = e_i$, where we recall that $(e_i)_{1 \leq i \leq m}$ are the canonical basis vectors of $\R^m$, with probability $\frac{1}{m}$ we know that
\begin{equation} \label{eq:munu1}
\mu = \frac{\lambda_{\min}(\mA)}{\trace{\mA}} \qquad \mbox{and} \qquad \nu = \frac{\trace{\mA}}{\min_{i=1,\ldots, m} \mA_{ii}} \enspace.
\end{equation}
Consequently, if the eigenvalues of $\mA$ are concentrated with $\lambda_{\min}(\mA)$ close to $\lambda_{\max}(\mA)$ then we have that $\mu \approx \frac{1}{m}$ and $\nu \approx m.$  Alternatively, if the eigenvalues of $\mA$  are far apart, then it may be that $\mu \approx 0$ and $\nu \approx \infty.$

\paragraph{No sketching.} When $\mS = \mI$ then $\mZ = \mA$ since $\mA$ is invertible. Consequently
\begin{equation}
	\label{eq:munum}
	\mu = \inf_{x \in \Range{\mA^\top}} \tfrac{\dotprod{\mA x,x}}{\dotprod{x,x}} = \lambda_{\min}(\mA), \qquad \qquad
	\nu = \sup_{x \in \Range{\mA^\top}} \tfrac{\dotprod{\E{\mA}x,x}}{\dotprod{\E{\mA}x,x}} =1 \enspace.
\end{equation}

In either of these two extremes, we need the smallest eigenvalue of $\mA$ to set $\mu$ and $\nu,$ which is a prohibitive cost. In~\Cref{sec:accelexp} we show that finding a setting for $\mu$ and $\nu$ that outperforms the standard sketch-and-project method is difficult, and akin to finding a needle in a haystack.

\subsection{Experiments setting the acceleration parameters}
\label{sec:accelexp}

Here we would like to verify if there exists a default setting for the acceleration parameters $\mu$ and $\nu$ that results in consistently faster execution than the non-accelerated version. In Figure~\ref{fig:grid_plot}, we show the results of an extensive grid search for trying to identify a suitable $\mu$ and $\nu$ setting.
This figure shows the time taken to reach a $\epsilon = 10^{-4}$ solution of the relative residual for different pairs of $\mu$ and $\nu$ such that
\begin{equation*}
	0 \; \leq \; \mu \; \leq \; \frac{1}{\nu} \; \leq  \; 1 \enspace.
\end{equation*}
The problem we considered here is kernel ridge regression, with a RBF kernel with $\sigma = 0.5$.
The data $\mX$ is a random $n \times d$ sparse CSC matrix with density $0.25$, and the regularizer is set to $\lambda = 1/n$.
From Figure~\ref{fig:grid_plot}, there is no clear pair of parameters $(\mu, \nu)$ leading to an improvement in convergence.
Even if a finer grid search might allow to find optimal parameters, the gain in convergence is so marginal that it makes acceleration impractical compared to the version without acceleration, see~\Cref{fig:accel_example}.

We conclude that there is yet no known way to set these acceleration parameters in practice: the theory might be too loose to set them and looking empirically on a grid search for optimal points is too cumbersome.

\begin{figure}
    \centering
    \begin{subfigure}{0.4\textwidth}
		\includegraphics[width=\textwidth]{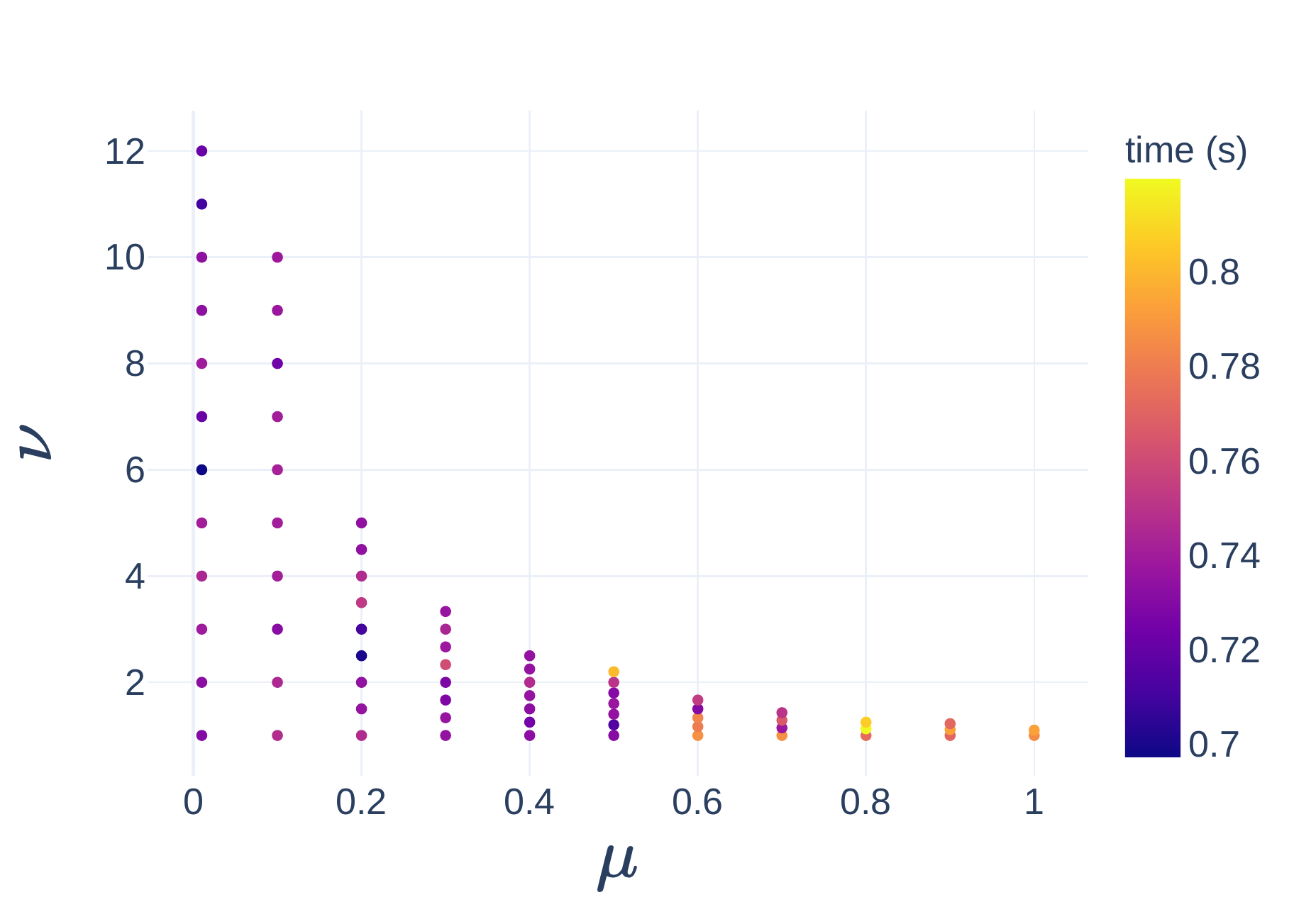}
		\caption{Average time taken over $100$ runs, for different acceleration parameters, to solve the problem with tolerance $10^{-4}$.}
		\label{fig:grid_plot}
    \end{subfigure}
	\hspace{0.02\textwidth}
    \begin{subfigure}{0.4\textwidth}
		\includegraphics[width=\textwidth]{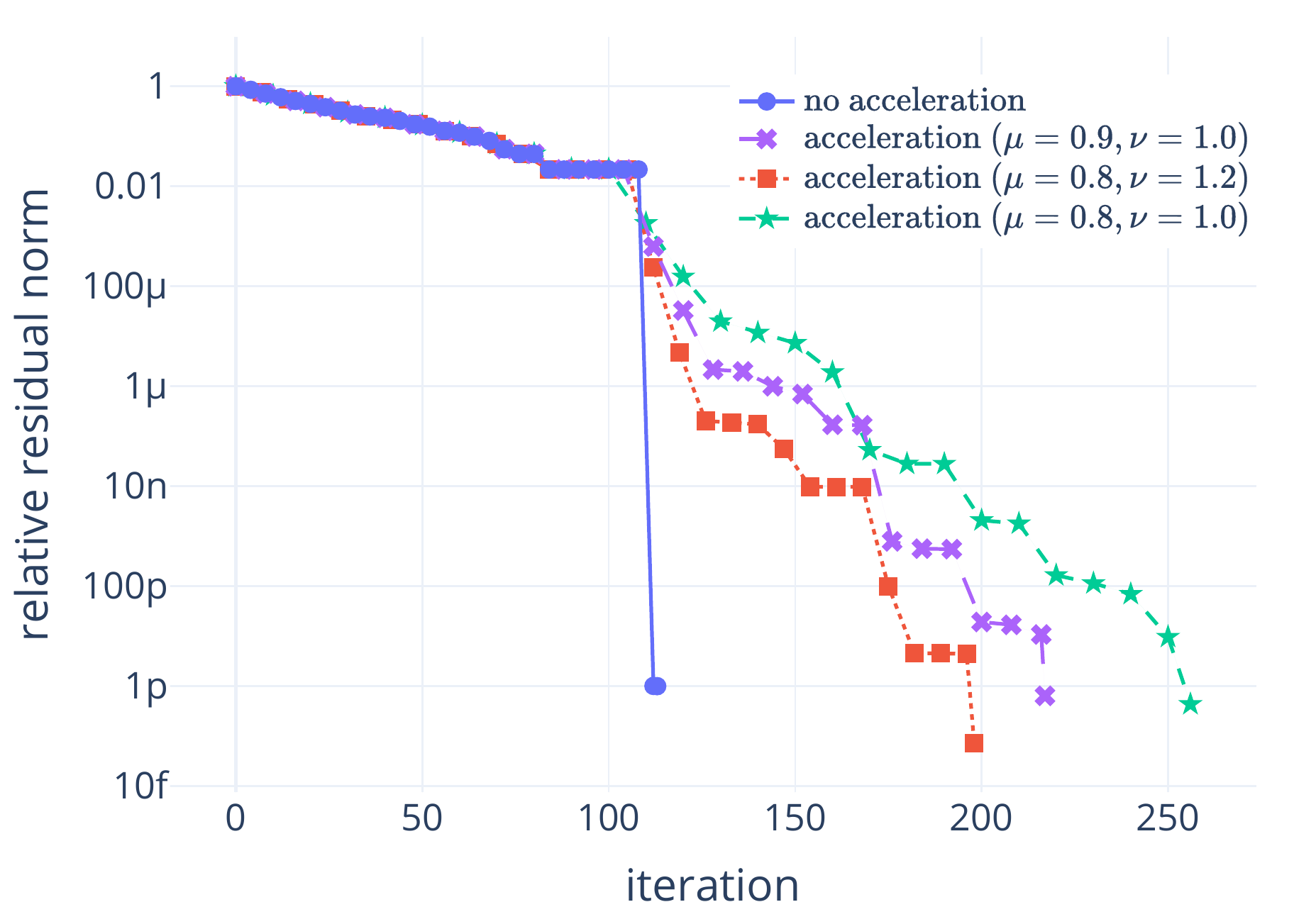}
        \caption{Comparison of no-acceleration and different acceleration parameter settings.}
		\label{fig:accel_example}
    \end{subfigure}
	\caption{Grid time plot (left) and single run (right) to solve a RBF kernel regression ($\sigma = 0.5$, $ \lambda = 1/n$) for random sparse CSC matrix $\mX \in \R^{n \times d}$ ($n=m=2000$, $d=500$) with density $0.25$. We used a \texttt{Subsample sketch} method and the	 sketch size is set to $\tau=\floor{m^{2/3}} = 158$.}
    \label{fig:acelexp}
\end{figure}

\renewcommand*{\bibfont}{}
{ 
\printbibliography
}

\newpage  

\appendix

\section{Auxiliary lemmas}

\begin{lemma}
	\label{lem:dual_equiv}
	The linear systems~\eqref{eq:linear} and~\eqref{eq:duallinear} have the same solution.
\end{lemma}
\begin{proof}
	Note that
	\begin{equation*}
		\left(\mX^\top \mX + \lambda \mI\right) \mX^\top = \mX^\top \left(\mX \mX^\top + \lambda \mI\right) \enspace.
	\end{equation*}
	Left multiplying the above by the inverse of $\left(\mX^\top \mX + \lambda \mI \right)$ and right multiplying by the inverse of $(\mX \mX^\top + \lambda \mI)$ gives
	\begin{equation*}
		\mX^\top \left(\mX \mX^\top +\lambda \mI\right)^{-1}= \left(\mX^\top \mX +\lambda \mI\right)^{-1} \mX^\top \enspace.
	\end{equation*}
	Finally right multiplying by $y$, on the left we have the solution to~\eqref{eq:duallinear} and on the right the solution to~\eqref{eq:linear}.
\end{proof}

\section{\texttt{RidgeSketch} package}
\label{sec:code}

Our \texttt{RidgeSketch} Python package is designed to be easily augmented by new contributions.
Users are encouraged to add new sketches, new parameter settings (\eg for momentum) and new datasets (see \path{datasets/data_loaders.py}).
They can also easily compare methods using the command:
\begin{lstlisting}[language=bash]
	$ python benchmarks.py [options] [name of config]
\end{lstlisting}
\texttt{RidgeSketch} comes with two tutorial Jupyter Notebooks: one for fitting a \texttt{RidgeSketch} model, and another for adding new sketches and benchmarks.
Next we present some snippets of code.

\subsection{Solving the ridge regression problem}

In Code~\ref{listing:code_intro}, we provide an example creating a ridge regression model and solving the fitting problem with our \texttt{Subsample sketch} solver.
\begin{listing}
\begin{minted}{python}
from datasets.data_loaders import BostonDataset
from ridge_sketch import RidgeSketch

dataset = BostonDataset()
X, y = dataset.load_X_y()  # data loader

model = RidgeSketch(
    alpha=1.0,  # regularizer
    algo_mode="mom",  # momentum version
    solver="subsample",  # sketch, cg, direct or sklearn solver
    sketch_size=10)
model.fit(X, y)  # solve the fitting problem
\end{minted}
\caption{Example of how to build a \texttt{RidgeSketch} model and fit it to a dataset.}
\label{listing:code_intro}
\end{listing}

\subsection{Adding sketching methods}

We built our package so that new sketching methods can easily be added by contributors. To do this, one should open the \texttt{sketching.py} file and create a new sketching subclass that inherits from the \texttt{Sketch} class.
In Code~\ref{listing:code_subsample} we provide the example of the \texttt{Subsample} sketching method.
\begin{listing}
\begin{minted}{python}
import numpy as np

class SubsampleSketch(Sketch):
    def __init__(self, A, b, sketch_size):
	"""Initializes the Sketch object and defines its attributes"""
        super().__init__(A, b, sketch_size)
        self.sample_indices = None

    def set_sketch(self):
        """Generates subsampling indices representing the sketching matrix"""
        self.sample_indices = np.random.choice(
            self.m, size=self.sketch_size, replace=False).tolist()

    def sketch(self, r):
	"""Generates sketched system"""
        self.set_sketch()
        SA = self.A.get_rows(self.sample_indices)
        SAS = SA[:, self.sample_indices]
        rs = r[self.sample_indices]
        return SA, SAS, rs

    def update_iterate(self, w, lmbda, step_size=1.0):
	"""Updates the weights by solving the sketched system"""
        w[self.sample_indices] -= step_size * lmbda
        return w
\end{minted}
\caption{Implementation of the \texttt{Subsample sketch} subclass.}
\label{listing:code_subsample}
\end{listing}

\end{document}